\documentclass[11pt]{article}
 \usepackage{amsmath, amsthm, amssymb, bbm, setspace,bigints}
 \usepackage[margin=1 in]{geometry}
\usepackage{caption}
\usepackage{subcaption}
\usepackage[toc,page]{appendix}
\usepackage{cite}         
\usepackage[pdftex]{graphicx}
\usepackage{pdfpages}
\usepackage{epstopdf}
\usepackage{booktabs}
\usepackage{authblk}

\usepackage{amsthm,bbm,amssymb}
\usepackage{amsmath}
\usepackage{natbib}
\usepackage[colorlinks,citecolor=blue,urlcolor=blue,filecolor=blue,backref=page]{hyperref}
\usepackage{graphicx}

\usepackage[utf8]{inputenc}
\usepackage{amsmath,amssymb,natbib,xcolor,multicol,url,mhequ}
\usepackage{graphicx,bbm,xr}
\usepackage{booktabs,epstopdf,color}
\usepackage[space]{grffile}
\usepackage{lineno}
\usepackage{multirow}

\def\mr{\mathrm}
\newcommand \zetaK{\zeta_{_K}}
\def\T{{ \mathrm{\scriptscriptstyle T} }}
\def\st{{ \mathrm{\scriptscriptstyle st} }}
\def\MF{{ \mathrm{\scriptscriptstyle MF} }}
\def\ELBO{{ \mathrm{\textstyle ELBO} }}

\def\m{\mathcal}
\def\mb{\mathbb}
\def\mr{\mathrm}
\def\wt{\widetilde}

\def\ind{\mathbbm{1}}

\def\T{\mathrm{\scriptscriptstyle{T}}}

\DeclareMathOperator{\TV}{TV}

\DeclareMathOperator*{\argmax}{arg\,max}

\def\m{\mathcal}
\def\mb{\mathbb}
\def\wt{\widetilde}
\def\T{{\mathrm{\scriptscriptstyle T} }}

\def\mr{\mathrm}
\def\ind{\mathbbm{1}}
\newcommand \bbP{\mathbb{P}}
\newcommand \bbE{\mathbb{E}}
\newcommand{\be}{\begin{equs}}
\newcommand{\ee}{\end{equs}}

\numberwithin{equation}{section}
\theoremstyle{plain}
\newtheorem{thm}{Theorem}[section]
\newtheorem{lemma}{Lemma}[section]

\newtheorem{corollary}{Corollary}[section]
\newtheorem{prop}{Proposition}[section]
\newtheorem{remark}{Remark}[section]

\usepackage{lipsum}
\usepackage{amsfonts,amsmath}
\usepackage{tikz-cd}
\usepackage{epstopdf}
\usepackage{algorithm} 
\usepackage{algorithmic}  
\usepackage[algo2e,ruled,vlined]{algorithm2e}

\title{Evidence bounds in singular models: probabilistic and variational perspectives 
}

\author[1]{Anirban Bhattacharya\thanks{anirbanb@stat.tamu.edu}}
\author[1]{Debdeep Pati\thanks{debdeep@stat.tamu.edu}}
\author[1]{Sean Plummer\thanks{splumme@tamu.edu}}

\affil[1]{Department of Statistics,  Texas A\&M University, College Station, Texas, 77843, USA}

\begin{document}
\maketitle

\begin{abstract}
The marginal likelihood or evidence in Bayesian statistics contains an intrinsic penalty for larger model sizes and is a fundamental quantity in Bayesian model comparison. Over the past two decades, there has been steadily increasing activity to understand the nature of this penalty in singular statistical models, building on pioneering work by Sumio Watanabe. Unlike regular models where the Bayesian information criterion (BIC) encapsulates a first-order expansion of the logarithm of the marginal likelihood, parameter counting gets trickier in singular models where a quantity called the real log canonical threshold (RLCT) summarizes the effective model dimensionality. In this article, we offer a probabilistic treatment to recover non-asymptotic versions of established evidence bounds as well as prove a new result based on the Gibbs variational inequality. In particular, we show that mean-field variational inference correctly recovers the RLCT for any singular model in its canonical or normal form. We additionally exhibit sharpness of our bound by analyzing the dynamics of a general purpose coordinate ascent algorithm (CAVI) popularly employed in variational inference. 

\end{abstract}


\section{Introduction}
Let $X^{(n)} = (X_1, \ldots, X_n)'$ denote $n$ independent and identically distributed observations from a 
probability density function $f(\cdot\mid \theta^\star)$. A Bayesian analysis in this setting proceeds by setting up (i) a statistical model consisting of a family of probability distributions $\{p(\cdot \mid \xi):  \xi \in \Omega\}$ for the individual observations, indexed by a parameter $\xi$ taking values in the parameter space $\Omega \subseteq \mb R^d$, and (ii) a prior (probability) distribution $\varphi(\cdot)$ on $\Omega$. The posterior distribution is given by
\be\label{eq:post}
\Pi(\xi \mid X^{(n)}) =  \frac{e^{\ell_n(\xi)} \varphi(\xi)}{m(X^{(n)})},  \quad \ell_n(\xi) :\,= \sum_{i=1}^n \log p(X_i \mid \xi),
\ee 
with $\ell_n(\xi)$ the log-likelihood function. The marginal likelihood or evidence 
\be \label{eq:marg}
m(X^{(n)}) = \int_{\Omega} e^{\ell_n(\xi)} \varphi(\xi) d\xi
\ee
is a fundamental object in Bayesian model comparison \citep{robert2007bayesian}, which encapsulates an intrinsic penalty for model complexity, and can be readily used to compare models with different parameter dimensions. However, barring conjugate settings this integral is rarely available in closed-form, necessitating approximate methods. 

A classical approach is to make analytic approximations, of which the Laplace approximation \citep{schwarz1978estimating,tierney1986accurate,kasstierkad90} is the most prominent. In {\em regular} parametric models, under mild assumptions, the Laplace approximation to the marginal likelihood takes the form 
\be\label{eq:Lap}
\log m(X^{(n)}) = \ell_n(\widehat{\xi}_n) - \frac{d \log n}{2} + R_n, 
\ee
where $\widehat{\xi}_n$ is the maximum likelihood estimate for $\xi$ based on $X^{(n)}$, $d$ is the parameter dimension, and the remainder term $R_n$ is bounded in magnitude by a constant free of $n$ with high probability. The quantity $2(\log \ell_n(\widehat{\xi}_n) - d \log n/2)$ is the celebrated Bayesian information criterion (BIC). 

The usual notion of a regular statistical model entails $\xi \mapsto p(\cdot  \mid \xi)$ is one-one and the Fisher information matrix $\bbE [(\partial^2 / \partial \xi^2 ) \log f( \cdot \mid \xi)]$ is positive definite for all $\xi \in \Omega$. In this article, our focus will be on {\em singular} statistical models, where at least one of the conditions for regularity are not met. Some common examples of singular models include mixture models, factor models, hidden Markov models, latent class analysis, neural networks etc. to name a few; see \cite{drton2017bayesian} for a more comprehensive list. As a simple concrete illustration, suppose $p(x \mid \xi) = \alpha \,\m N(x; 0, 1) + (1 - \alpha) \m N(x; \mu, 1)$ with $\xi = (\alpha, \mu) \in [0, 1] \times \mb R$. The map $\xi \mapsto p(\cdot \mid \xi)$ is clearly not one-one as the entire region $\Omega_0 :\, = \{1\} \times \mb R \cup [0, 1] \times \{0\}$ inside the parameter space get mapped to the $\m N(0, 1)$ distribution. The Fisher information matrix is also not positive definite on $\Omega_0$. 

The derivation of the Laplace approximation proceeds by localizing the integral \eqref{eq:Lap} to a neighborhood of the maximum likelihood estimate (or the posterior mode) and subsequently applying a second-order Taylor series expansion of the log-likelihood around $\widehat{\xi}_n$ to reduce the integral Eq.\,\eqref{eq:Lap} to a Gaussian integral. It should perhaps then be intuitive that this approximation will face difficulties for singular models where the Hessian matrix can be singular. This is indeed the case and can be verified via simulation in a straightforward manner; see, e.g., the instructive Example 1 of \cite{drton2017bayesian}. However, finding the precise asymptotic behavior of the marginal likelihood for general singular models is a highly non-trivial exercise. The foundational groundwork for a general theory has been laid in a series of seminal contributions by Watanabe \citep{watanabe1999algebraic,watanabe2001balgebraic,watanabe2001algebraic}, with much of the subsequent development condensed into book-level treatments in \cite{watanabe2009algebraic,watanabe2018mathematical}. We also refer the reader to Shaowei Lin's thesis \citep{lin2011algebraic} and the background section of \cite{drton2017bayesian} for lucid summaries of this beautiful theory. 

Watanabe shows that in singular settings, a more general version of the Laplace approximation is given by
\be\label{eq:wata}
\log m(X^{(n)}) = \ell_n(\xi^\star) - \lambda \log n + (m-1) \log (\log n) + R_n,
\ee
assuming that the data is generated from $p(\cdot \mid \xi^\star)$. The stochastic error term $R_n$ is $O_{P^\star}(1)$ as before. The quantity $\lambda \in (0, d/2]$ is called the {\em real log-canonical threshold} (RLCT) and the integer $m \ge 1$ its {\em multiplicity}. 
Only when $\lambda = d/2$ and $m = 1$, one recovers the usual Laplace approximation as a special case of the expansion \eqref{eq:wata}. However, in general, the usual Laplace approximation no longer provides a correct approximation to the log evidence. For more on model selection in singular settings, we refer the reader to \cite{watanabe2013widely,drton2017bayesian}. 
Over the years, there has been a growing literature on determining (or bounding) $\lambda$ for specific singular statistical models; see 
\cite{yamazaki2003singularities,rusakov2005asymptotic,aoyagi2005stochastic,aoyagi2010stochastic,hayashi2017tighter,drton2017tree,aoyagi2019learning} for a flavor of this literature. 

Watanabe's derivation of Eq.\,\eqref{eq:wata} has two major ingredients. First, the parameter space is partitioned and parameter transformations are performed to express the integrand in Eq.\,\eqref{eq:marg} over each partition to a more manageable {\em normal crossing} (or simply, normal) form. The existence of such partitions and parameter transformations is guaranteed by a famous result in algebraic geometry due to Hironaka \citep{hironaka1964resolution} on the resolution of singularities. Watanabe then analyzes the asymptotic order of a generic integrand in normal form using complex analytic tools and Schwartz distribution theory \citep{friedlander1998introduction}. The RLCT and its multiplicity have simple analytically tractable expressions for an integral in normal form; see \S\,2 for the exact details.

In this article, we revisit the general problem of estimating an integral in normal form. Our primary motivation behind this work was to explore the possibility of deriving Eq.\,\eqref{eq:marg} exclusively using probabilistic arguments readily accessible to the wider statistics and machine learning audience. We approach this from two distinct angles -- one using more conventional arguments such as stochastic ordering and conditioning, while the other hinging on the Gibbs variational inequality. As a by-product of the probabilistic treatment, all our results are non-asymptotic in nature. We carry out the first part of this program in \S\,\ref{sec:probab}. We follow standard practice to first analyze a deterministic version of the problem, replacing the log-likelihood ratio with its expectation under the data generating model, and then proceed to handle the stochastic component. Interestingly, the RCLT and its multiplicity appear as the rate and shape parameters of a certain Gamma distribution in our analysis. 

Variational approaches \citep{mackay2003information,bishop2006pattern,wainwright2008graphical} have increasingly grown in popularity in Bayesian statistics as a different set of probabilistic tools to approximate the evidence. Variational Bayes (VB) aims to find the best approximation to the posterior (or another target) distribution from a class of tractable probability distributions, with the approximation error most commonly measured in terms of a Kullback--Leibler divergence. This scheme equivalently produces a lower bound to the log-marginal likelihood, commonly known as the evidence lower bound (ELBO). One of the most popular choices for the approximating class of distributions is the mean-field family constituting of product distributions, whose origins can be traced back to statistical physics \citep{parisi1988statistical}. The mean-field approximation has seen enormous applications in Bayesian statistics due to its simplicity as well as availability of general purpose coordinate ascent algorithms (CAVI; \cite{bishop2006pattern}) to approximate the optimal ELBO. 

In \S\,\ref{sec:var}, we show that mean-field variational inference correctly recovers the RLCT for normal forms, even though the posterior distribution itself has strong dependence and is far from a product structure (see Figure \ref{fig:ex12} for an example). To show this result, we first produce a candidate solution from the mean-field class which provides the correct order of the ELBO up to $\log(\log n)$ terms. Next, by analyzing the dynamics of the aforesaid coordinate ascent algorithms in the 2d case, we establish that the order of the ELBO at the candidate solution can not be globally improved, hence showing our bound is sharp. Studying the dynamics of the algorithm was also instrumental in guiding us towards an analytic form of the candidate solution. While asymptotics of the ELBO for mean-field VB have been studied in specific models such as mixture models \citep{watanabe2004gaussian, watanabe2005exponential, watanabe2006stochastic, watanabe2007mixture,  watanabe2007generalized}, hidden Markov models \citep{hosino2005vhmm}, stochastic context-free grammars \citep{hosino2006grammar}, and Boltzmann machines \citep{watanabe2009bipartite}, the general result proven here is new to the best of our knowledge. Our analysis adds to the emerging literature on algorithmic behavior of mean field VB \citep{zhang2017theoretical,mukherjee2018mean,ghorbani2018instability,plummer2020dynamics}. Beyond the Bayesian statistics literature, we were also inspired by the recent success of mean-field approximations to estimate the normalizing constant for probabilistic graphical models \citep{chatterjee2016nonlinear,Basak2017,austin2019structure,yan2020nonlinear}.

\section{Nonasymptotic probabilistic bounds for normal form}\label{sec:probab}
We begin with introducing some notation. We reserve the notations $\bbE^\star$ and $\bbP^\star$ to respectively denote expectation and probability under (the $n$-fold product of) $p(\cdot \mid \xi^\star)$, where $\xi^\star$ denotes the true data generating parameter. Let $K_n(\xi) = n^{-1} \, [\ell_n(\xi^\star) - \ell_n(\xi)]$ 
be the negative log-likelihood ratio scaled by a factor of $n^{-1}$, so that its $\bbE^\star$-expectation is the Kullback--Leibler divergence,
\be 
K(\xi) := \bbE^\star \big( K_n(\xi) \big) = D\big( p(\cdot \mid \xi^\star) \, \| \, p (\cdot \mid \xi) \big). 
\ee 
Unlike regular models, the set $\{\xi: K(\xi) = 0\}$ contains more than one point for singular models. Define 
\be\label{eq:marglik}
\m Z(n) = \int_\Omega e^{-n K_n(\xi)} \varphi(\xi), \quad \m Z_K(n) = \int_\Omega e^{-n K(\xi)} \varphi(\xi). 
\ee
It is immediate that $\m Z(n) = \log m(X^{(n)}) - \ell_n(\xi^\star)$, and it is equivalent to study the asymptotic behavior of $\m Z(n)$ with $n$. The deterministic quantity $\m Z_K(n)$ is closely related to $\m Z(n)$ as it is obtained by replacing the stochastic quantity $K_n(\xi)$ with its expectation $K(\xi)$ under the true distribution. Let us denote
\be\label{eq:Z_i}
Z_i(\xi) =  \log \bigg\{ \frac{ p(X_i \mid \xi^\star)}{p(X_i \mid \xi)} \bigg\} - \bbE^\star \log \bigg\{ \frac{ p(X_i \mid \xi^\star)}{p(X_i \mid \xi)} \bigg\}, \quad i = 1, \ldots, n,
\ee
so that $n^{-1} \sum_{i=1}^n Z_i(\xi) = [K_n(\xi) - K(\xi)]$ characterizes the difference between $K_n$ and $K$ as an average of i.i.d. random variables.  

\

\noindent {\bf Normal-crossing form.} Throughout the paper, we assume $K(\xi) = \xi^{2 \mr{k}} :\,= \xi_1^{2k_1} \ldots \xi_d^{2k_d}$ is a monomial with $\mr{k} = (k_1, \ldots, k_d)^\T \in \mb N^d$ a multi-index having at least one positive entry; and that the prior density $\varphi(\xi) = b(\xi) \, \xi^{\mr{h}}$, where $\mr{h} = (h_1, \ldots, h_d)^\T \in \mb N^d$ is another multi-index and $b(\cdot) > 0$ is a real analytic function on $\Omega$.
This setting is referred to as a {\em normal crossing form} or simply {\em normal form}. While these choices may seem very specific, they in fact completely encapsulate the complexity of the general problem. This impactful observation was made by Watanabe based on a deep result in algebraic geometry due to Hironaka on the resolution of singularities, and played a major role in the development of singular learning theory. A simplified form of Hironaka's theorem from Chapter 6 of \cite{watanabe2018mathematical} is quoted below with minor notational changes. 
\begin{thm}[Hironaka's theorem]\label{thm:Hironaka} Assume that $K(\xi) \geq 0$ is a nonzero analytic function on $\Omega$ and that the set $\{\xi \in \Omega: K(\xi) = 0\}$ is not empty. Then there exist $\epsilon > 0$, sets $\{\Xi_j; \Xi_j \subset \Omega \}$ and $\{U_j; U_j \subset \mb R^d\}$ such that $\{\xi \in \Omega: K(\xi) < \epsilon\} = \bigcup_j \Xi_j$, 
and, for each pair $(\Xi_j, U_j)$, there exist analytic maps $g_j : U_j \mapsto \Xi_j$ satisfying
\begin{eqnarray*}
K(g_j(u)) = u^{2\mr{k}_j}, \quad |g_j'(u)| = b_j(u)|u^{\mr{h}_j}|, 
\end{eqnarray*}
where $|g_j'(u)|$ denotes the absolute value of the determinant of the Jacobian matrix $J = (\partial \xi_i/\partial u_l)_{il}$
of the transformation $\xi = g_j(u)$. Moreover, $b_j(u) > 0$ for all $j$ and $\mr{k}_j, \mr{h}_j$ are multi-indices. 
\end{thm}
For a given $K(\cdot)$, the  theorem guarantees the existence of the coordinate maps $\{g_j\}$ under which $K$ can be locally identified with a monomial on each $U_j$. Hence, the overall integral is first expressed as the sum of integrals over each $\Xi_j$, and within each $\Xi_j$, a parameter transformation is made using the map $g_j$ to the reduce the integral to a normal form. 

The normal form offers a convenient reduction since the real log canonical threshold $\lambda$ and its multiplicity $m$ for normal forms are determined by the multi-indices $\mr{k}$ and $\mr{h}$ in a particularly simple fashion: $\lambda$ is the minimum of the numbers $\{(h_j + 1)/(2 k_j)\}_{j=1}^d$ and $m$ is the number of indices $j$ which assume the minimum value. For example, in the $d = 2$ case, the general theory implies
\be 
\int_{[0, 1]^2} e^{-n \xi_1^2 \xi_2^2 } d\xi \asymp \frac{ C \log n}{\sqrt{n}}, \quad \int_{[0, 1]^2} e^{-n \xi_1^2 \xi_2^4 } d\xi \asymp \frac{C}{n^{1/4}}, 
\ee
since in the first case, $\mr{k} = (1, 1)^\T$ and $\mr{h} = (0, 0)^\T$, implying $\lambda = \min\{1/2, 1/2\} = 1/2$ with multiplicity $m = 2$; while in the second one, $\mr{k} = (1, 2)^\T$ and $\mr{h} = (0, 0)^\T$, implying $\lambda = \min\{1/2, 1/4\} = 1/4$ and $m = 1$.

\begin{figure}[htbp!]
    \centering
    \includegraphics[scale=0.48]{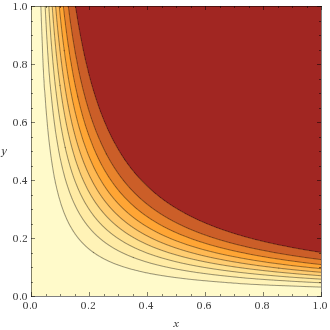}
      \includegraphics[scale=0.48]{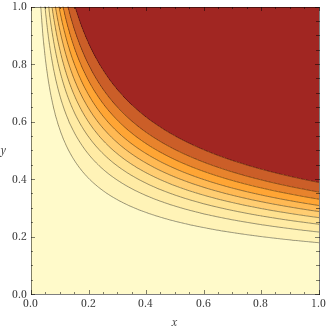}
    \caption{Contour plot of $\exp(-n x^2 y^2)$ (left) and $\exp(-n x^2 y^4)$ (right) on $[0, 1]^2$ for $n = 100$}
    \label{fig:ex12}
\end{figure}

As a concrete statistical example that we shall repeatedly return to in this article, we consider Example 27 in \cite{watanabe2018mathematical} pertaining to a single-layer neural network model, where response and covariate pair $(y, x) \in \mb R \times [0, 1]$ have a joint density modeled as $p(y, x \mid \theta) = p(y \mid x, \theta) \, p(x) = \m N(y; \theta_1 \tanh(\theta_2 x), 1) \, \ind_{[0, 1]}(x)$, with $\theta = (\theta_1, \theta_2)^\T \in [0, 1]^2$. Suppose the true parameter is $(0, 0)^\T$, and assume a uniform prior on $\theta$. Watanabe shows that 
\be 
K(\theta_1, \theta_2) = \frac{\theta_1^2 \theta_2^2}{2} \, K_0(\theta_2), \quad K_0(t) = \int_0^1 \frac{\tanh^2(tx)}{t^2} \, dx. 
\ee
Moreover, under the transformation, 
\be \label{eq:trans_neunet}
\xi_1 = \theta_1, \quad \xi_2 = \theta_2(K_{0}(\theta_2)\slash2)^{1\slash 2}:\,=g(\theta_2),
\ee
the model reduces to the normal form with $(k_1, k_2) = (1, 1)$ and $(h_1, h_2) = (0, 0)$, implying $\lambda = 1/2$. While a single resolution map suffices in this example, this, however, may not be the case in general.

Based on the above discussion, we work under the assumption that $\Omega = [0, 1]^d$, $K(\xi) = \xi^{2 \mr{k}}$, and $\varphi(\xi) = b(\xi) \, \xi^{\mr{h}}$ for the rest of the article. We now proceed to derive non-asymptotic bounds to the integrals in Eq.\,\eqref{eq:marglik} using probabilistic arguments. We first analyze the non-stochastic quantity $\m Z_K(n)$ in \S\,\ref{ssec:prob} and treat $\m Z(n)$ in \S\,\ref{ssec:2ndorder}.  

\subsection{The deterministic quantity $\m Z_K(n)$}\label{ssec:prob}

In this subsection, we take up the analysis of the non-stochastic quantity $\m Z_K(n)$. Watanabe used a number of powerful complex analytic tools to study the asymptotic behavior of $\m Z_K(n)$ as $n \to \infty$. The asymptotic behavior of $\m Z_K(n)$ is dictated by the Laurent expansion of the associated complex-valued {\em zeta function}
\be 
\zetaK(z) = \int K(\xi)^{-z} \, \varphi(\xi) d\xi, \quad z \in \mb C. 
\ee
In particular, if $(\lambda, m)$ is the smallest pole and its multiplicity of the meromorphic function $\zetaK$, then 
\be 
\m Z_K(n) \approx C n^{-\lambda} \, (\log n)^m. 
\ee
Chapter 3 of \cite{lin2011algebraic} contains an exposition on the Laurent expansion of $\zetaK$. Alternatively, one may recognize $\m Z_k(n)$ as the Laplace transform of a quantity called the state-density function, which is a generalized function in the parlance of Schwartz distribution theory. The state-density function and the zeta function are inter-related, with the zeta function being the Mellin transform of the state-density function. See Chapter 4 of \cite{watanabe2009algebraic} and Chapter 5 of \cite{watanabe2018mathematical} for a derivation of the asymptotics of $\m Z_K(n)$ based on the state-density function. 

Our goal here is to provide a non-asymptotic two-sided bound to $\m Z_K(n)$ for normal forms based entirely on basic probabilistic arguments. Interestingly, the quantities $\lambda$ and $m$ turn out to be related to the rate and shape parameters of a collection of gamma densities, as we shall see below. In the first result, we assume $b(\xi) = 1$ and treat the general case as a corollary. 

\begin{thm}\label{thm:rclt_nform}
Let $K(\xi) = \xi^{2 \mr{k}}$ for $\xi \in \Omega = [0, 1]^d$ and $\mr{k} = (k_1, \ldots, k_d)^\T \in \mb N^d$ with at least one positive entry, and let $\varphi(\cdot)$ be a probability density on $\Omega$ with $\varphi(\xi) \,\propto\, \xi^{\mr{h}}$, where $\mr{h} = (h_1, \ldots, h_d)^\T \in (0, \infty)^d$. Then, there exists positive constants $C_1$ and $C_2$ independent of $n$ such that 
\be 
C_1 \frac{(\log n)^{m-1}}{n^{\lambda}} < \m Z_K(n) < C_2 \frac{(\log n)^{m-1}}{n^{\lambda}},
\ee
where 
$$ 
\lambda = \min_{j} \frac{h_j + 1}{2 k_j}, \quad m = \sum_{j=1}^d \ind\bigg(\frac{h_j + 1}{2 k_j} = \lambda\bigg).
$$
\end{thm}

\begin{proof}
The main idea behind our proof is to exploit the natural representation of $\m Z_K(n)$ as the expectation of a random variable with respect to the prior measure. Specifically, let $T = K(\xi)$, where $\xi \sim \varphi$ is a random variable distributed according to the prior measure. Then, it immediately follows that the real random variable $T$ takes values in the unit interval $[0, 1]$ and $\m Z_K(n) = E e^{-n T}$. Before proceeding to simplify this expectation, we note some conventions and notation. Let $\bar{d} = \sum_{j=1}^d \ind(k_j \ne 0)$, and without loss of generality assume that $k_j > 0$ for $j = 1, \ldots, \bar{d}$ and $k_j = 0$ for $j > \bar{d}$. Define $\lambda_j :\, = (h_j + 1)/(2 k_j)$ for $j = 1, \ldots, \bar{d}$, and without loss of generality, further assume that these are sorted in non-decreasing order $\lambda_1 \le \lambda_2 \ldots \le \lambda_{\bar{d}}$. By definition, $\bar{d} \ge m$, and the first $m$ of the $\lambda_j$s all equal $\lambda$. Throughout, we use the convention that an $\mbox{Expo}(\beta)$ distribution has density $\beta e^{-\beta x} \ind_{(0, \infty)}(x)$, that is, $\beta$ denotes the rate parameter of the distribution.

The random variable $Z :\,= -\log T$ can be expressed as $Z = \sum_{j = 1}^{\bar{d}} Z_j$ with $Z_j = - \log (\xi_j^{2 k_j})$ for $j = 1, \ldots, \bar{d}$. An application of the change of measure formula yields that $Z_j \sim \mbox{Expo}(\lambda_j)$ with $\lambda_j = (h_j + 1)/(2 k_j)$ as defined above; interestingly, observe the quantities $(h_j + 1)/(2 k_j)$s appear as the exponential rate parameters. Moreover, since the prior measure $\varphi$ has a product form, the $Z_j$s are independent across $j$. Letting $\Phi_K(\cdot)$ denote the cumulative distribution function of $T$, we then have, for any $t \in (0, 1)$, 
\be\label{eq:T_cdf}
\Phi_K(t) = P(T \le t) = P(- \log T \ge -\log t) = P\bigg(\sum_{j=1}^{\bar{d}} Z_j \ge \log(1/t)\bigg). 
\ee
It follows from the above display that $\lim_{t \downarrow 0} \Phi_K(t) = 0$, $\lim_{t \uparrow 1} \Phi_K(t) = 1$, and $\Phi_K$ is an absolutely continuous cdf that admits a density $\varphi_K(\cdot)$ with respect to the Lebesgue measure, given by, 
\be\label{eq:T_pdf}
\varphi_K(t) = \frac{1}{t} \, g_Z\big( \log(1/t) \big) \, \ind_{(0,1)}(t),  
\ee
where $g_Z$ is the density of $Z$ with respect to the Lebesgue measure. Our object of interest, 
\be\label{eq:Zkn_expr}
\m Z_K(n) = \int_{0}^1 e^{-nt} \varphi_K(t) \, dt = \int_0^n e^{-t} \, \frac{1}{t} \, g_Z\big( \log(n/t) \big) \, dt.
\ee
Before proceeding to prove the theorem in its entire generality, we consider two special cases which are instructive in themselves and also help build towards the general proof.  

\

First, consider the special case where $\lambda_j = \lambda$ for all $j = 1, \ldots, \bar{d}$. Then, $m = \bar{d}$ and $Z \sim \mbox{Gamma}(m, \lambda)$, where a $\mbox{Gamma}(\alpha, \beta)$ distribution has density $\big( \beta^{\alpha}/\Gamma(\alpha) \big) \, e^{-\beta x} x^{\alpha -1} \ind_{(0, \infty)}(x)$. It follows that for any $t \in (0, n)$, 
\be 
g_Z \big( \log(n/t) \big) = \frac{\lambda^m}{\Gamma(m)} \, n^{-\lambda} \, t^{\lambda} \big( \log(n/t) \big)^{m-1}.
\ee
Substituting in equation \eqref{eq:Zkn_expr}, we obtain that 
\be\label{eq:spcase1}
\m Z_K(n) = C n^{-\lambda} \, \underbrace{ \int_0^n t^{\lambda-1} \, e^{-t} \, \big( \log(n/t) \big)^{m-1} \, dt }_{\m I_m(n)} \asymp n^{-\lambda} \, (\log n)^{m - 1}.
\ee
The proof of the assertion that $\m I_m(n) \asymp (\log n)^{m - 1}$ for any $m \ge 1$ is straightforward and hence omitted. This completes the proof for this particular case. 

\

As a second special case, suppose $\lambda_1 < \ldots < \lambda_{\bar{d}}$, which implies that $\lambda = \lambda_1$ and $m = 1$. The distribution of $Z$ isn't recognizable as a standard density any longer, although an analytic expression for its density is available in the literature as quoted below. 
\begin{thm}[\citep{mathai1982storage,bibinger2013notes}]\label{thm:sumexp}
Let $Z_k \stackrel{ind.} \sim \mbox{Expo}(\lambda_k)$ for $k = 1, \ldots, K$, with $\lambda_1 < \ldots < \lambda_K$. Then, the density $g_Z$ of $Z = \sum_{k=1}^K Z_k$ is 
\be
g_Z(z) = \sum_{k=1}^K \, \underbrace{\bigg( \prod_{r \ne k} \frac{\lambda_r}{\lambda_r - \lambda_k} \bigg)}_{b_k} \, g_k(z),
\ee
where $g_k(z) = \lambda_k e^{-\lambda_k z}$ is the density of $Z_k$. 
\end{thm}
The coefficients $\{b_k\}$ can be both positive and negative, and thus the above is not a mixture of exponential densities. However, the coefficient $b_1$ corresponding to the smallest rate parameter $\lambda_1$ is positive. We have, for any $t \in (0, n)$, 
\be 
g_Z \big( \log(n/t) \big) = \sum_{j=1}^{\bar{d}} b_j \, \lambda_j \, n^{-\lambda_j} \, t^{\lambda_j}.
\ee
Substituting this expression in equation \eqref{eq:Zkn_expr}, we get
\be\label{eq:spcase2}
\m Z_K(n) = \sum_{j=1}^{\bar{d}} b_j \, \lambda_j \, n^{-\lambda_j} \int_0^n e^{-t} t^{\lambda_j - 1} dt \asymp \sum_{j=1}^{\bar{d}} b_j n^{-\lambda_j} \asymp n^{-\lambda_1}.
\ee
This proves the theorem for this special case. The fact that $b_1 > 0$ has been crucially used to arrive at the last conclusion in the above display, along with the fact that $n^{-\lambda} > n^{-\lambda'}$ for $\lambda' > \lambda > 0$. This example carries the takeaway message that the exact form of the density $g_Z$ is of secondary importance, and the focus should be on extracting the most significant contribution in terms of $n$. This is our strategy for the most general case. 
 
\

In the general case, assume that there are $d^\ast \le \bar{d}$ unique $\lambda$-values $\lambda_1^\ast < \lambda_2^\ast \ldots < \lambda_{d^\ast}^\ast$ among $\{\lambda_j\}_{j=1}^{\bar{d}}$ with corresponding multiplicities $m_1, \ldots, m_{d^\ast}$. It is then immediate that $\sum_{s=1}^{d^\ast} m_s = \bar{d}$. Also, $(\lambda_1^\ast, m_1) = (\lambda, m)$ from the theorem statement. Exploiting the independence of the $Z_j$s, we write $Z = \sum_{s=1}^{d^\ast} W_s$, with $W_s \stackrel{ind.} \sim \mbox{Gamma}(m_s, \lambda_s^\ast)$ for $s = 1, \ldots, d^{\ast}$. While there exist expressions for the density of sum of independent Gamma random variables \citep{mathai1982storage}, they are much more cumbersome than the simpler case of exponentials in Theorem \ref{thm:sumexp}. Hence, we do not attempt to work with the density $g_Z$ and instead aim to bound $\m Z_K(n)$ from both sides. To that end, we crucially use the idea of stochastic ordering of random variables. 

Recall that for real random variables $X_1, X_2$, $X_1$ is said to be stochastically smaller than $X_2$ if for every $x \in \mb R$, $P(X_2 > x) \ge P(X_1 > x)$. We use the notation $X_1 <_{\st} X_2$ to denote this stochastic ordering. We now record a useful result. 
\begin{lemma}\label{lem:stoc_order}
Consider the random variable $Z = \sum_{s=1}^{d^\ast} W_s$, with $W_s \stackrel{ind.} \sim \mathrm{Gamma}(m_s, \lambda_s^\ast)$. Assume $\lambda_1^\ast < \ldots < \lambda_{d^\ast}^\ast$ and let $\bar{d} = \sum_{s=1}^{d^\ast} m_s$. Define $Z_\ell = W_1$ and $Z_c = \sum_{s=2}^{d^\ast} \wt{W}_s$, where $\wt{W}_s \stackrel{ind.} \sim \mathrm{Gamma}(m_s, \lambda_2^\ast)$ are also independent of $W_1$. Then, $Z_\ell \sim \mathrm{Gamma}(m_1, \lambda_1^\ast)$, $Z_c \sim \mathrm{Gamma}(\bar{d} - m_2, \lambda_2^\ast)$, $Z_\ell$ and $Z_c$ are independent, and with $Z_u :\, = Z_\ell + Z_c$, 
\be 
Z_\ell <_{\st} Z <_{\st} Z_u. 
\ee
\end{lemma}
With this result in place, we now aim to bound $\m Z_K(n) = E e^{-n T}$. Since $e^{-n T}$ is a non-negative random variable taking values in $(0, 1)$, we have 
\be 
\begin{aligned}\label{eq:id}
E e^{-n T} 
& = \int_{u=0}^1 P\big(e^{-nT} > u\big) du \\
& = \int_{u=0}^1 P\big(T < \log(1/u)/n\big) du = \int_{u=0}^1 P\big(Z > \log n - \log(\log 1/u)\big) du. 
\end{aligned}
\ee
For any $z > 0$, we have the following two-sided inequality from Lemma \ref{lem:stoc_order},
\be 
P(Z_\ell > z) < P(Z > z) < P(Z_u > z) < P(Z_\ell > z) + P(Z_c > z). 
\ee
Here, the last inequality follows from an application of the union bound. Substituting this inequality at the end of equation \eqref{eq:id} for every $u$ and working backwards, we obtain
\be 
E e^{- n T_\ell} < E e^{-n T} < E e^{-n T_\ell} + E e^{-n T_c},
\ee
where $T_\ell = e^{- Z_\ell}$ and $T_c = e^{-Z_c}$. Since $Z_\ell$ and $Z_c$ are both gamma random variables, it follows from equation \eqref{eq:spcase1} that $E e^{- n T_\ell} > C n^{-\lambda} (\log n)^{m-1}$ and $E e^{- n T_\ell} + E e^{-n T_c} < C_1 n^{-\lambda} (\log n)^{m-1} + C_2 n^{-\lambda_2} (\log n)^{m_2-1} < C_3 n^{-\lambda} (\log n)^{m-1}$. This delivers the desired bound.

\end{proof}

We now state a corollary to Theorem \ref{thm:rclt_nform} relaxing the assumption on the prior
\begin{corollary}\label{cor:genprior}
Assume the setup of Theorem \ref{thm:rclt_nform}.  Let $b: U \to \mb R$ be an analytic function with $b(\mathbf{0}) \ne 0$, where $U$ is any open subset of $\mb R^d$ containing $\Omega$. Then, 
\be 
\int_{\Omega} b(\xi) \, e^{-n K(\xi)} \, \varphi(\xi) \, d\xi \asymp n^{-\lambda} (\log n)^{m-1}.
\ee
\end{corollary}
Corollary \ref{cor:genprior} shows that the assumption of a product prior in Theorem \ref{thm:rclt_nform} can be relaxed to more general priors of the form $\wt{\varphi}(\xi) \,\propto\, b(\xi) \varphi(\xi)$, with the same asymptotic order of the normalizing constant as before.

\subsection{The stochastic quantity $\m Z(n)$}\label{ssec:2ndorder}
We now extend our probabilistic analysis from the previous subsection to analyze the stochastic quantity $\m Z(n)$. Write
\begin{eqnarray*}
\m Z(n) = \int_\Omega e^{-n K(\xi) - \sqrt{n} \, \xi^{\mathrm{k}} W_n(\xi)} \varphi(\xi) d\xi,
\end{eqnarray*}
where 
\be \label{eq:W_n}
W_n(\xi) = \frac{1}{\sqrt{n}} \, \sum_{i=1}^n \xi^{-\mathrm{k}} \, Z_i(\xi).  
\ee
We make some simplifying assumptions to keep the presentation from getting notationally too heavy. We shall assume $\varphi(\xi) \,\propto\, \xi^{\mathrm{h}}$, and also that $\lambda_j = \lambda$ for all $j = 1, \ldots, \bar{d} < d$. Recall from \S\,\ref{ssec:prob} that $m = \bar{d}$ in this case. Denote by $I$ the set  $\{1, \ldots, m\}$ and $J= \{m+1, \ldots, d\}$.  Clearly, under $\varphi(\cdot)$, the common distribution of the independent random variables  $\xi_j$ for $j \in I$ is $\mbox{Beta}(h_j +1, 1)$ and hence $\xi_j^{2k_j}$ is $\mbox{Beta}(\lambda, 1)$ distributed for $j \in I$.

We now state a stochastic approximation result for $\m Z(n)$ in Theorem \ref{prop:sec_order} that can be considered as a non-asymptotic version of Theorem 11 in \cite{watanabe2018mathematical}. The main idea lies in decoupling the effect of the singular part $\xi_I$ controlled by $K(\xi)$ from the non-singular $\xi_J$ part of $\xi$.  As we shall see in Theorem \ref{prop:sec_order}, our proof relies heavily on the properties of the conditional density of  $\xi_I$ given $K(\xi)$. 

Since the distribution of $Z = -\log T = -\sum_{j=1}^m 2k_j \log \xi_j$ is $\mbox{Gamma}(m, \lambda)$, the conditional distribution  $(- 2k_1\log \xi_1, \ldots, -2k_m\log\xi_m) := (Z_1, \ldots, Z_m) \mid Z$ is given by $Z \times \mbox{Dirichlet}(\mathbf{1}_m)$ and is expressed as 
\begin{eqnarray*}
f_{Z_1, \ldots, Z_m \mid Z}(z_1, \ldots, z_m) = \frac{\Gamma(m)}{Z^{m-1}}, \quad 0 \leq \sum_{i=1}^{m-1} z_i \leq Z, \quad z_m = Z-  \sum_{i=1}^{m-1} z_i. 
\end{eqnarray*}
Hence, the conditional density of $\xi \mid Z = (e^{-Z_1/(2k_1)}, \ldots, e^{-Z_m/(2k_m)})\mid  Z$ is given by 
{\small \begin{eqnarray}\label{eq:cond}
\varphi_{\xi \mid Z}(\xi_I) = \frac{2^m\prod_{j=1}^m k_j}{\prod_{j=1}^m \xi_j}  \frac{\Gamma(m)}{Z^{m-1}}, \quad e^{-Z}\leq (\xi_{I_{-m}})^{2k_{-m}} \leq 1, \quad   \xi_I^{2k} = Z,
\end{eqnarray}}
where $I_{-m} = I \backslash \{m\}, k_{-m} = k \backslash \{k_m \}$. Also, $\varphi_{\xi \mid Z=z}(\xi_I)$ is the same as the density $\varphi_{\xi \mid T= e^{-z}}(\xi_I)$.  Define a sequence of stochastic processes $D_n(t, \xi)$ with index set $\mathbb{R}^+ \times [0, 1]^d$ as 
\be
D_n(t, \xi) =  t^{\lambda -1} e^{- t- \sqrt{t} W_n(\xi)} \varphi_{\xi \mid Z= -\log (t/n)} (\xi_I) \varphi(\xi_J). 
\ee
Further, define an integrated version of $D_n(t, \xi)$ as $D_n(t) = \int_\Omega D_n(t, \xi) d\xi$ for $t \in \mb R^+$.  
\begin{thm} \label{prop:sec_order}
The following expression provides a non-asymptotic stochastic expansion of $\m Z(n) $ with $n^{-\lambda} (\log n)^{m-1}$ as the leading term,  
\be
\m Z(n) \prod_{j=1}^d(h_j+1) =  \frac{n^{-\lambda} (\log n)^{m-1} \lambda^m}{\Gamma(m)}  \int_0^n  D_n(t)dt+ R_n
\ee
where 
\be
R_n =  \int_{t=0}^n r_n(t) D_n(t) dt, \,  r_n(t)= \frac{\lambda^m}{\Gamma(m)} \,  n^{-\lambda} \, \sum_{j=1}^{m -2} {m-1 \choose j} \big( \log n \big)^j (- \log t)^{m-1-j}.
\ee

Moreover, the remainder term $R_n$ is smaller order in comparison with the dominating term. 
If the sequence of stochastic processes $W_n$ satisfied $\|W_n\|_\infty = O_p(1)$, then 
\be
\frac{|R_n|}{n^{-\lambda} (\log n)^{m-1}}\to 0, 
\ee 
 almost surely.

\end{thm}

\begin{proof}
{\bf First part:} 
By abuse of notation we shall assume that $\varphi$ corresponds to a product Beta density $\prod_{j=1}^d \mbox{Beta}(\xi_j \mid h_j+1, 1)$.  Multiplying $Z(n)$ by $\prod_{j=1}^d(h_j+1)$ we have
\be
\m Z(n) \prod_{j=1}^d(h_j+1)  &= \int e^{-n K(\xi) - \sqrt{n} \xi^k W_n(\xi)} \varphi(\xi) d\xi \\
&= \int_0^1 e^{-nt} \Big[\int e^{-\sqrt{nt} W_n(\xi)} 
  \varphi_{\xi \mid T=t}(\xi_I) \varphi(\xi_J) d\xi \Big] \varphi_K(t) dt
\ee
where $\varphi_K(t)$ is the density of $T = K(\xi)$ as in \S \ref{ssec:prob}.
Substituting  $\varphi_K(t) = (1/t) g_Z\big\{ \log (1/t) \big\}$, where 
$g_Z(\cdot)$  is the pdf of a $\mbox{Gamma}(\lambda, m)$ random variable, we have by another change of variable, $nt \mapsto t$ that, 
\be
\m Z(n) \prod_{j=1}^d(h_j+1) 
&= \int_0^n e^{-t}  \frac{1}{t} g_Z\Big( \log \frac{n}{t}\Big) \Big[\int e^{-\sqrt{t} W_n(\xi)} 
  \varphi_{\xi \mid T=t/n}(\xi_I) \varphi(\xi_J) d\xi \Big] dt. 
\ee
Noting, 
\be 
g_Z \big( \log(n/t) \big) = \frac{\lambda^m}{\Gamma(m)} \, n^{-\lambda} \, t^{\lambda} \big( \log(n/t) \big)^{m-1} 
=  \frac{\lambda^m}{\Gamma(m)} \,  n^{-\lambda} \, t^{\lambda} \big( \log n \big)^{m-1} + r_n(t) t^{\lambda}
\ee
it follows
\be
\m Z(n) \prod_{j=1}^d(h_j+1) 
&=  \frac{\lambda^m}{\Gamma(m)} n^{-\lambda} \, \big( \log n \big)^{m-1} \int_0^n \int D_n(t, \xi) d\xi  dt + R_n. 
\ee

%
\noindent {\bf Second part:}  Since the maximum exponent of all logarithmic terms  inside the summation is $(m-2)$ we have 
\be
\frac{|r_n(t)|}{n^{-\lambda} (\log n)^{m-1}} \leq C(m) (\log n)^{-1} \sum_{j=1}^{m-2}|\log t|^{m-1-j}
\ee
for some constant $C(m)$ depending on $m$.  Also
 $D_n(t) \leq t^{\lambda-1} e^{-t + \sqrt{t} \| W_n\|_\infty}$ and hence 
 \be
\frac{|R_n|}{n^{-\lambda} (\log n)^{m-1}} \leq    \frac{C(m)}{\log n } \sum_{j=1}^{m-2} \int_{t=0}^\infty e^{-t +\sqrt{t} \| W_n\|_\infty} t^{\lambda-1} |\log t|^{m-1-j}  dt
\ee
Since the function $e^{-t +\sqrt{t} \| W_n\|_\infty} t^{\lambda-1} |\log t|^{m-1-j}$ is integrable and $\|W_n\|_\infty = O_p(1)$, the result follows immediately.

\end{proof}
\subsubsection{Connections with Watanabe's result:}
%
%
Watanabe \citep{watanabe2018mathematical} arrives at an asymptotic expansion of 
$\m Z(n)$ using a different technique. He refers to the inverse Laplace transform of the integrand $\exp(-nK(\xi)) \varphi(\xi)$ as the {\em state density function} given by  $\delta(t - \xi^{2k})  \varphi(\xi)$, where $\delta(\cdot)$ is the Dirac delta function in the sense of a generalized function in Schwartz distribution theory. Thereafter, 
\be
\exp(-nK_n(\xi)) \varphi(\xi) = \int_{t=0}^1 e^{-nt + \sqrt{nt} W_n(\xi)} \delta(t - \xi^{2k})  \varphi(\xi) dt
\ee
whence the marginal likelihood  $\m Z(n)$ becomes
\be 
\m Z(n)  = \int \int_{t=0}^1 e^{-nt  + \sqrt{nt} W_n(\xi)}  \delta(t - \xi^{2k}) dt  \varphi(\xi) d\xi := \int_{0}^n v(t) dt,
\ee
with $v(t) = \int e^{-t  + \sqrt{t} W_n(\xi)}  \delta(t/n- \xi^{2k}) d\xi$. 
Watanabe then distills $v(t)$ by passing it through a Mellin's transform $\int v(t) t^z dt$, discards the smaller order terms before transforming back to $\tilde{v}(t) \asymp n^{-\lambda} t^{\lambda -1} \{\log (n/t)\}^{m-1} e^{-t  + \sqrt{t} W_n(\xi)}$ using the inverse Mellin transform.  Substituting back into $\m Z(n)$,
\be 
\m Z(n) \asymp  \int_{0}^n \tilde{v}(t) dt 
\asymp n^{-\lambda} (\log n)^{m-1}\int_{0}^n t^{\lambda -1}e^{-t  + \sqrt{t} W_n(\xi)} dt. 
\ee

Our proof technique avoids the Mellin's transform and instead relies on simple probabilistic tools. The main differences are as follows:
\begin{itemize}
\item The use of Dirac delta function as a generalized function is avoided by taking the conditional distribution of $\xi$ with respect to $\xi^{2k}$ and multiplying with the marginal density of $\xi^{2k}$. 

\item The conditioning neutralizes the effect of the singular part along with the stochastic component, while the marginal density provides the overall order. 

\item No approximation is made during the process. In contrast, while taking the inverse Mellin's transform, smaller order terms are dropped. Hence our representation of $\m Z(n)$  in Theorem \ref{prop:sec_order}  is exact as opposed to Theorem 11 in \cite{watanabe2018mathematical}.
\end{itemize}

Now, we show that asymptotically as $n \to \infty$,  Theorem  \ref{prop:sec_order} recovers Theorem 11 in \cite{watanabe2018mathematical}.   An important ingredient  of making this connection is to show a weak convergence of the sequence of stochastic processes $W_n$.   The expected value and the covariance of $Z_i(\xi)$ are
$  \mathbb{E} Z_i(\xi)=0,   \mbox{cov}[Z_i(\xi), Z_i(\zeta)]:=c_z(\xi,\zeta) $, respectively.  
For $W_n(\xi)$, the same quantities are given by
\begin{align*}
    \mathbb{E} W_n(\xi)= \frac{\xi^{-k}}{\sqrt{n}}\sum_{i=1}^n\mathbb{E} Z_i(\xi)=0, \quad
    \mbox{cov}[W_n(\xi), W_n(\zeta)]= \frac{c_z(\xi,\zeta)}{\xi^k\zeta^k}:=c_w(\xi,\zeta). 
\end{align*}
Let $W^*$ denote a mean zero Gaussian process on $\Omega = [0, 1]^d$ with covariance kernel  $c_w$.  Under appropriate conditions on the stochastic processes $\{\tilde{Z}_i(\xi) = \xi^{-k}Z_i(\xi): \xi \in [0, 1]^d\}$, we show in Proposition \ref{prop:weak} that $W_n$ weakly converges to  $W^*$.  The proof requires sub-Gaussianity \citep{vershyninhigh} of $\tilde{Z}_i(\xi)$. \footnote{ A real valued random variable $X$ is called sub-Gaussian if there exists a constant $c_X > 0$ such that $\mb P(|X| > t) \leq 2 e^{-t^2/(2c_X)}, \quad t \geq 0$.}

\noindent {\bf Assumption A1:} 
Suppose that $\tilde{Z}_i(\xi)$ are iid sub-Gaussian. Furthermore suppose that there exists a positive function $L: \mathbb{R} \to \mathbb{R}^+$ with $\mb Ee^{t L(X_1)} \leq e^{t^2/(2c_L)}$ for every $t > 0$ and for some constant $c_L > 0$, such that 
\be
\vert \tilde{Z}_i(\xi)- \tilde{Z}_i(\xi')\vert \leq  L(X_i) \|\xi - \xi'\|. 
\ee
\begin{prop}\label{prop:weak}
 Under  Assumption {\bf A1}, 
 $W_n \stackrel{w}{\rightarrow}W^*$. 
\end{prop}
Sections 10.4 and 10.5 of \cite{watanabe2018mathematical}  provide heuristic arguments to study weak convergence of $W_n$; the assumptions require  $ \xi^{-k}Z_i(\xi)$ to be at least $d/2+1$ times differentiable. On the other hand, our {\bf Assumption A1} requires  $\tilde{Z}_i(\xi)$ to be Lipschitz and sub-Gaussian. In many examples, such as the aforementioned one-layered neural network in Example 27 of \cite{watanabe2018mathematical}, this is straightforward to verify. In this example, 
\be
Z_i(\theta_1,\theta_2)=\frac{1}{2}\theta_1^2\theta_2^2F^2(X_i; \theta_2) - \theta_1\theta_2Y_iF(X_i; \theta_2) + K(\theta_1,\theta_2),
\ee
where $F(x;\theta_2):=\tanh(\theta_2x)\slash \theta_2$.  Then,
\be
\tilde{Z}_i(\theta_1,\theta_2) =\frac{Z_i(\theta_1,\theta_2)}{\theta_1\theta_2}=\frac{1}{2}\theta_1\theta_2F^2(X_i; \theta_2) - Y_iF(X_i; \theta_2) + \frac{1}{2}\theta_1\theta_2K_0(\theta_2). 
\ee
Applying the change of variables $\xi_1=\theta_1$ and $\xi_2= \theta_2(K_0(\theta_2)\slash \theta_2)^{1\slash 2}=g(\theta_2)$ and noting that 
 $\tilde{Z}_i(\xi) := \tilde{Z}_i(\xi_1,g^{-1}(\xi_2))$, we have 
\be 
 \tilde{Z}_i(\xi) =\frac{1}{2}\xi_1g^{-1}(\xi_2)F^2\{X_i; g^{-1}(\xi_2)\} - Y_iF\{X_i; g^{-1}(\xi_2)\} +   \frac{1}{2}\xi_1g^{-1}(\xi_2)K_0\{g^{-1}(\xi_2)\}.
\ee
We show  $\tilde{Z}_i(\xi)$ above satisfies {\bf Assumption A1} in Appendix \ref{sec:A1verify}. 

\

The next ingredient in establishing the connection is to establish a weak limit of $D_n(t) = \int D_n(t, \xi) d\xi$.  In  Proposition \ref{lem:Dnt_conv}, we show that for each $t > 0$, $D_n(t)$ converges weakly to the following fixed stochastic process
\be
D(t)
=  \int t^{\lambda -1}e^{-t - \sqrt{t} W^*({\bf 0}, \xi_J)}  \varphi(\xi_J) d\xi_J. 
\ee
In addition,  we also show in Proposition \ref{lem:Dnt_conv} that  $\int_0^n D_n(t)dt$ converges in distribution to $\int_0^\infty D(t)dt$.
 \begin{prop}\label{lem:Dnt_conv} 
If $W_n \stackrel{w}{\rightarrow} W^*$, $\|W^*\|_\infty = O_p(1)$ and  $\xi_I \mapsto W^*(\xi_I, \xi_J)$ is  almost surely continuous, then for each $t > 0$,
$D_n(t) \stackrel{w}{\rightarrow} D(t)$ and  
 $\int_{0}^nD_n(t) dt  \stackrel{w}{\rightarrow}\int_0^\infty D(t)dt$.     
\end{prop}
Using Theorem \ref{prop:sec_order} and Proposition \ref{lem:Dnt_conv},   
\be
\m Z(n) \prod_{j=1}^d(h_j+1) &\sim  \frac{n^{-\lambda} (\log n)^{m-1} \lambda^m}{\Gamma(m)}  \int_0^\infty t^{\lambda -1}  e^{-t + \sqrt{t}W^*({\bf 0}, \xi_J)}   \varphi(\xi_J) \varphi(\xi_J) dt \nonumber \\
\m Z(n)&\sim  \frac{n^{-\lambda} (\log n)^{m-1} }{2^m(m-1)!\prod_{j=1}^m k_j}  \int_0^\infty t^{\lambda -1}  e^{-t + \sqrt{t}W^*({\bf 0}, \xi_J)}  \xi_J^{h_J} dt, \label{eq:watacompare}
\ee
where $h_J =(h_{m+1}, \ldots, h_d$). 
Using the properties of the Dirac delta function, we can write  
\be
D(t)
=  \int  t^{\lambda -1} e^{-t -\sqrt{t} W^*(\xi)} \delta_{{\bf 0}}(\xi_I) \varphi(\xi_J) d\xi. 
\ee
where  $ \delta_{{\bf 0}}(\xi_I)$ is a Dirac delta measure at ${\bf 0}$, which also appears in Theorem 11 in \cite{watanabe2009algebraic} in the expansion of $\m Z(n) $.   Observing that that $k_j = 0$ for $j=m+1, \ldots d$, \eqref{eq:watacompare} exactly matches with the equation (5.32) of Theorem 10 or the expression under Theorem 11 in \cite{watanabe2009algebraic}.

\section{Mean-field VI for normal form}\label{sec:var}

In this section, our goal is to show that mean-field variational approximation always correctly recovers the RLCT for singular models in normal form, which therefore constitute an interesting class of statistical examples where the mean-field approximation is provably better than the Laplace approximation. An important ingredient of our analysis is to study the dynamics of the associated CAVI algorithm, which might be of independent interest. While mean-field inference is known to produce meaningful parameter estimates in many statistical models \citep{yang2017alpha,pati2017statistical}, the algorithmic landscape contains both positive and negative results \citep{wang2006convergence,zhang2017theoretical,mukherjee2018mean,ghorbani2018instability}.

To simplify our analysis, we shall work with the deterministic quantity $\m Z_K(n)$. Define a probability density 
\be\label{eq:VIdetpost}
\gamma_K^{(n)}(\xi) = \frac{e^{-n K(\xi)} \, \varphi(\xi)}{\mathcal{Z}_K(n)}, \quad \xi \in \Omega,
\ee
with $K(\xi) = \xi^{2 \mr{k}}$ in normal form as in Theorem \ref{thm:rclt_nform} and $\varphi(\xi) \,\propto\, b(\xi) \xi^{\mr{h}}$ is a probability density on $\Omega$ with the analytic function $b(\cdot)$ as in Corollary \ref{cor:genprior}. Clearly, $\mathcal{Z}_K(n)$ is then recognized as the normalizing constant of $\gamma_K^{(n)}(\cdot)$, which serves as a deterministic version of the posterior defined in Eq.\,\eqref{eq:post}. 

For any probability measure $\rho$ on $\Omega$ with $\rho \ll \varphi$, the following well-known identity is easy to establish, 
\be\label{eq:VIid}
D\big(\rho \,\|\, \gamma_K^{(n)} \big) = \log \m Z_k(n) + \bigg[ \int_\Omega n K(\xi) \, \rho(d \xi) + D(\rho \,\|\, \varphi)\bigg], 
\ee
where recall $D(\mu \,\|\, \nu) = E_\mu (\log d\mu/d\nu)$ is the Kullback--Leibler divergence between $\mu$ and $\nu$. An immediate upshot of this is the Gibb's variational inequality, which states that for any probability density $\rho \ll \varphi$ on $\Omega$, 
\be\label{eq:GibbsVI}
\log \m Z_K(n) \ge \Psi_n(\rho) \,:= - \bigg[ \int n K(\xi) \, \rho(d \xi) + D(\rho \,\|\, \varphi)\bigg],
\ee
with equality attained if and only if $\rho = \gamma_K^{(n)}$. The Gibb's variational inequality is central to a variational approximation to the normalizing constant $\mathcal{Z}_K(n)$. The quantity $\Psi_n(\rho)$ in the right hand side of \eqref{eq:GibbsVI} is a lower bound to $\log \mathcal{Z}_K(n)$ for any $\rho \ll \varphi$. A variational lower bound to $\log \mathcal{Z}_K(n)$ is then obtained by optimizing the variational parameter $\rho$ over a family of probability densities $\mathcal{F}$ on $\Omega$, 
\be\label{eq:ELBO}
\log \mathcal{Z}_K(n) \ge \ELBO({\m F}) :\, = \sup_{\rho \in \m F} \, \Psi_n(\rho). 
\ee
The notation $\ELBO$ here abbreviates {\em evidence lower bound}, which is commonly used to designate the variational lower bound in Bayesian statistics. If the supremum in \eqref{eq:ELBO} is attained at some $\rho^\star \in \m F$, the density $\rho^\star$ is called the optimal variational approximation. It follows from equation \eqref{eq:VIid} that $\rho^\star$ is a best approximation to $\gamma_K^{(n)}$ in terms of KL divergence from the class $\m F$, i.e., 
$$
D\big(\rho^\star \,\|\,  \gamma_K^{(n)} \big) = \inf_{\rho \in \m F} D\big(\rho \,\|\, \gamma_K^{(n)} \big). 
$$
The choice of the family $\m F$ typically aims to balance computational tractability and expressiveness. A popular example is the mean-field family, 
\be\label{eq:MF}
    \mathcal{F}_{\MF} :\, = \big\{ \rho = \rho_1 \otimes \ldots \otimes \rho_d \,:\, \rho \ll \varphi \text{ a prob. measure on } \Omega\big\},
\ee
where $\rho$ is assumed to be a product-measure, with no further restriction on the constituent arms. Mean-field variational approximation has its origins in statistical physics \citep{parisi1988statistical}, and has subsequently found extensive usage in Bayesian statistics for parameter estimation and model selection \citep{bishop2006pattern}. 

Since $\gamma_K^{(n)}$ does not lie in $\m F_{\MF}$ for any $n$, it follows that the inequality in equation \eqref{eq:ELBO} is a strict one if we restrict $\m F$ to the mean-field family. We, however, show below that the mean-field approximation correctly recovers the leading order term in the asymptotic expansion of $\log \m Z_K(n)$.


\begin{thm}\label{thm:VB_main}
Consider a variational approximation \eqref{eq:ELBO} to $\m Z_K(n)$ in equation \eqref{eq:VIdetpost}, where the variational family $\m F$ is taken to be the mean-field family $\m F_{\MF}$ defined in equation \eqref{eq:MF}. Then, there exists a constant $C$ independent of $n$ such that $\ELBO(\m F_{\MF}) \ge - \lambda \log n - C$, where $\lambda = \min_{j} [(h_j + 1)/(2 k_j)]$ is the real log canonical threshold. 
\end{thm}
Since $\log \m Z_K(n) \asymp - \lambda \log n + (m-1) \log (\log n)$, it follows that the mean-field approximation correctly recovers the leading order term in the asymptotic expansion of $\log \m Z_K(n)$. This, in particular, implies that the relative error $R_n$ due to the mean-field approximation 
\be
    R_n :\,= \frac{ | \log \m Z_K(n) - \ELBO(\mathcal{F}_{\MF}) | }{ | \log \m Z_K(n)  |} \to 0 \text{ as } n \to \infty. 
\ee

This is rather interesting, as the density $\gamma_K^{(n)}$ clearly does not lie in $\m F_{\MF}$ for any finite $n$. 
\begin{proof}
Our strategy is to produce a candidate $\wt{\rho} = \otimes_{j=1}^d \wt{\rho}_j$ such that $\Psi_n(\wt{\rho}) \ge -\lambda \log n - C$. We introduce some notation before describing the $\wt{\rho}_j$s. For $k, h, \beta > 0$, define a density $f_{k, h, \beta}$ on $[0, 1]$ given by 
\be\label{eq:qkhb}
    f_{k, h, \beta}(u) = \frac{ u^h \exp(-\beta u^{2k}) \mathbbm{1}_{[0, 1]}(u) }{ B(k, h, \beta) }, 
\ee
where $B(k, h, \beta) = \int_0^1 x^h \exp(-\beta x^{2k}) dx$. We record two useful facts about $f_{k, h, \beta}$ in the Lemma below. We collect some well-known facts first about the incomplete gamma function; see \cite{abramowitz1964handbook}. 
\begin{remark}\label{rem:incomp_gamma}
For $x, a > 0$,  denote the lower incomplete gamma function by
\be
\gamma(a,x) = \frac{1}{\Gamma(a)} \int_0^x t^{a-1}e^{-t} \,dt. 
\ee
For any fixed $a > 0$, $\gamma(a,\cdot)$ takes values in $(0, 1)$, with $\lim_{x \to \infty} \gamma(a,x) = 1$. Also, $\lim_{x \to 0} \gamma(a, x) \slash x^a = 1 \slash \Gamma(a+1)$, and 
\be 
\gamma(a+1,x) = \gamma(a,x) - \frac{x^a e^{-x}}{\Gamma(a+1)}. 
\ee
\end{remark}

\begin{lemma}\label{lem:qkhb}
Let the density $f_{k,h, \beta}$ be as in equation \eqref{eq:qkhb}. Then, \\[1ex]
(i) The normalizing constant $B(k,h,\beta)$ is given by 
\be
B(k, h, \beta) = \frac{\beta^{-\lambda}\Gamma(\lambda)\gamma(\lambda, \beta)}{2k}.
\ee
(ii) The quantity $\int_0^1 u^{2k} f_{k,h, \beta}(du)$ depends on $k$ and $h$ only through $\lambda :\,= (h+1)/(2k)$. Call this expectation $G(\lambda, \beta)$, and we have 
\be
    G(\lambda, \beta) :\,= \int_0^1 u^{2k} f_{k,h, \beta}(u)du = \frac{\lambda}{\beta} \, \frac{\gamma(\lambda+1,\beta)}{\gamma(\lambda, \beta)}.
\ee
(iii) We have, 
\be 
\lim_{\beta \to \infty} \frac{ |\log B(k, h, \beta) - (-\lambda \log \beta)| }{ \lambda \log \beta} = 0, \quad \lim_{\beta \to \infty} \frac{G(\lambda, \beta)}{\lambda \slash \beta} = 1. 
\ee
Thus, for large $\beta$, $\log B(k, h, \beta) \asymp - \lambda \log \beta$, and $G(\lambda, \beta) \asymp \lambda \slash \beta$.
\end{lemma}

We now construct $\wt{\rho} = \otimes_{j=1}^d \wt{\rho}_j$. Let $g \in \{1, \ldots, d\}$ be such that $(h_g + 1)/(2 k_g) = \lambda$; in case there are multiple such indices, we arbitrarily break tie. Set $\wt{\rho}_j = f_{k_j, h_j, \beta_j}$ for $j = 1, \ldots, d$ with $\beta_g = n$ and $\beta_j = 1$ for $j \ne g$. With this choice, let us now bound $\Psi_n(\wt{\rho})$. We divide this up into two parts. First, we have
\be
    \int_\Omega n K(\xi) \wt{\rho}(\xi) d\xi = n \prod_{j=1}^d \int_0^1 \xi_j^{2 k_j} \wt{\rho}_j(\xi_j) d\xi_j = n \prod_{j=1}^d G(\lambda_j, \beta_j) = n G(\lambda, n) \, \prod_{j \ne g} G(\lambda_j, 1), 
\ee
where recall that $\lambda_j = (h_j + 1)/(2 k_j)$. For the first equality, we use the product form of both $K(\cdot)$ and $\wt{\rho}$. The second inequality uses Lemma \ref{lem:qkhb} and that $(\lambda_g, \beta_g) = (\lambda, n)$. The quantity $\prod_{j\ne g} G(\lambda_j, 1)$ is clearly a constant free of $n$, and from part (iii) of Lemma \ref{lem:qkhb}, $n G(\lambda, n) \asymp \lambda$. This implies $\int n K(\xi) \wt{\rho}(\xi) d\xi$ is overall of a constant order. 

Next, consider $D(\wt{\rho} \,\|\, \varphi)$. Let $\wt{\varphi}$ be the probability density on $\Omega$ with $\bar{\varphi}(\xi) \,\propto\, \xi^{\mr{h}}$. Write 
\be
    D(\wt{\rho} \,\|\, \varphi) = D(\wt{\rho} \,\|\, \bar{\varphi}) + \int_\Omega \wt{\rho} \, \log \frac{ \bar{\varphi} }{ \varphi }. 
\ee
The second term in the above display, up to an additive constant, is $- \int_\Omega \log b(\xi) \, \wt{\rho}(d\xi)$, which is bounded above by $- \log b_1$ since $b(\cdot) > b_1$ on $\Omega$. Hence, we focus attention on the first term $D(\wt{\rho} \,\|\, \bar{\varphi})$, which is simpler to analyze than $D(\wt{\rho} \,\|\, \varphi)$ since both $\wt{\rho}$ and $\bar{\varphi}$ have a product form as in equation \eqref{eq:MF}. In particular, we have $D(\wt{\rho} \,\|\, \bar{\varphi}) = \sum_{j=1}^d D(\wt{\rho}_j \,\|\, \bar{\varphi}_j)$, where $\bar{\varphi}_j$ is the $j$th marginal of $\bar{\varphi}$ with density 
$\bar{\varphi}_j(u) \,\propto \, u^{h_j}$ for $u \in [0, 1]$. Since $\wt{\rho}_j = f_{k_j, h_j, \beta_j}$, we have,
\be
    D(\wt{\rho}_j \,\|\, \bar{\varphi}_j) 
& = - \beta_j \int_0^1 u^{2 k_j} f_{k_j, h_j, \beta_j}(u) du - \log B(k_j, h_j, \beta_j) + \log(1+h_j) \nonumber\\
& = - \beta_j G(\lambda_j, \beta_j) - \log B(k_j, h_j, \beta_j) + \log(1+h_j). 
\ee
If $j \ne g$, the above quantity is some constant free of $n$. For $j = g$, from part (iii) of Lemma \ref{lem:qkhb}, we get $D(\wt{\rho}_g \,\|\, \bar{\varphi}_g) = C + \lambda \log n$. 

Putting the pieces together, we have proved that $\Psi_n(\wt{\rho}) \ge -\lambda \log n - C$ for some constant $C$ free of $n$. This completes the proof.

\end{proof}

\subsection{Coordinate ascent algorithm and analysis of dynamics}
In this section, we first provide some insight into our choice of the candidate solution $\wt{\rho}$ in the proof of Theorem \ref{thm:VB_main}. The choice of $\wt{\rho}$ was motivated by empirically analyzing the behavior of a coordinate ascent (CAVI) algorithm for the optimization problem $\sup_{\rho \in \m F_{\MF}} \Psi_n(\rho)$ in the $d = 2$ case, which naturally constrains the coordinate updates to lie in the family of densities $\{f_{k, h, \beta} \}$. Secondly, we comment on the ``optimality" of the candidate $\wt{\rho}$, regarding which Theorem \ref{thm:VB_main} is inconclusive. In particular, the theorem only obtains a lower bound to $\ELBO(\m F_{\MF})$, and it is natural to question whether there is a scope for improvement -- note that there is a $\log(\log n)$ gap between the asymptotic order of $\log \m Z_K(n)$ and the lower bound to the ELBO. Studying the dynamics of the CAVI algorithm, we demonstrate a class of examples where $\ELBO(\m F_{\MF}) = - \lambda \log n +C$ for some constant $C$, implying no further improvement is possible uniformly. 

The Coordinate Ascent Variational Inference (CAVI) algorithm is popular in statistics and machine learning for maximizing an evidence lower bound over a mean-field family; see Chapter 10 of Bishop  for a book-level treatment. The CAVI can be interpreted as a cyclical coordinate ascent algorithm which at any iteration $t \ge 1$ cycles through maximizing $\Psi_n(\rho)$ as a function of $\rho_j$, keeping $\{\rho_\ell\}_{\ell \ne j}$ fixed at their current value $\{\rho_\ell^{(t)}\}_{\ell \ne j}$. For example, in the $d = 2$ case, the iterates $\rho^{(t)} = \rho_1^{(t)} \otimes \rho_2^{(t)}$ for $t \ge 1$ are given by 
\be 
\rho_1^{(t)} = \argmax_{\rho_1} \Psi_n\big(\rho_1 \otimes \rho_2^{(t-1)}\big), \quad \rho_2^{(t)} = \argmax_{\rho_2} \Psi_n\big(\rho_1^{(t)} \otimes \rho_2\big), 
\ee
with an arbitrary initialization $\rho^{(0)} = \rho_1^{(0)} \otimes \rho_2^{(0)} \ll \varphi$, and assuming the first component gets updated first. The objective function $\Psi_n(\rho_1 \otimes \rho_2)$ is concave in each argument\footnote{although, it is rarely jointly concave} so that the maximization problems in the update step above have unique solutions. Moreover, these maximizers admit a convenient integral representation, which facilitates tractability of the updates in conditionally conjugate models. It is straightforward to see that the successive CAVI iterates increase the objective function value, since for any $t \ge 1$, 
\be\label{eq:cavi_gen}
\Psi_n\big(\rho_1^{(t)} \otimes \rho_2^{(t)}\big) \ge \Psi_n\big(\rho_1^{(t)} \otimes \rho_2^{(t-1)}\big) \ge \Psi_n\big(\rho_1^{(t-1)} \otimes \rho_2^{(t-1)}\big), 
\ee
although convergence to a global optimum is not guaranteed in general. 

Returning to the present case, consider the normal form of a singular model with parameter dimension $d = 2$, 
\be \label{eq:special_2d}
\gamma_K^{(n)}(\xi_1, \xi_2) \,\propto\, \xi_1^{h_1}\xi_2^{h_2} \, \exp(-n\xi_1^{2k_1}\xi_2^{2k_2}), \quad (\xi_1, \xi_2) \in [0,1]^2,
\ee
resulting from setting $b(\xi) \equiv 1$ in equation \eqref{eq:VIdetpost}. Let $\lambda_i=(h_i+1)\slash 2k_i$ for $i=1,2$ as usual; we assume without loss of generality that $\lambda_1\leq \lambda_2$, implying the real log canonical threshold for this model is $\lambda=\lambda_1$. The mean field family $\m F_{\text{MF}}$ in this case consists of product distributions $\rho_1 \, \otimes \, \rho_2$, where $\rho_1$ and $\rho_2$ are absolutely continuous densities on $[0, 1]$. We derive in Appendix  \ref{app2} that the $t$th iteration of the CAVI algorithm \eqref{eq:cavi_gen} in this case is $\rho^{(t)}(\xi)=\rho_{1}^{(t)}(\xi_1) \cdot \rho_{2}^{(t)}(\xi_2)$, with
\be\label{eq:cavi_t}
\rho_{1}^{(t)}(\xi_1) = f_{k_1, h_1, n\mu_1^{(t)}}(\xi_1), \quad \rho_{2}^{(t)}(\xi_2) = f_{k_2, h_2,n\mu_2^{(t)}}(\xi_2),
\ee
where recall the density $f_{k, h, \beta}$ is defined in equation \eqref{eq:qkhb} and for $t \ge 1$,
\be\label{eq:mu_t}
\mu_{1}^{(t)} = G\big(\lambda_2, n \mu_2^{(t-1)}\big), \quad \mu_2(t) = G\big(\lambda_1, n\mu_1^{(t)}\big). 
\ee
We also record the value of the ELBO at iteration $t$, 
\be
     \Psi_n(\rho^{(t)}) = & -n \, G\big(\lambda_1, n \mu_1^{(t)} \big) \, G\big(\lambda_2, n\mu_2^{(t)} \big)  \nonumber \\
     & + n \, \mu_1^{(t)} \, G\big(\lambda_1,n \mu_1^{(t)} \big) +n \, \mu_2^{(t)} \, G\big(\lambda_2,n \mu_2^{(t)} \big) \nonumber\\
    & + \log B\big(h_1,k_1,n\mu_1^{(t)}\big) + \log B\big(h_2,k_2,n\mu_2^{(t)} \big).
 \label{eq:elbo_t}
\ee
Inspecting equation \eqref{eq:cavi_t}, it becomes apparent that the changes in $\rho_i^{(t)}$ across $t$ entirely takes place through the univariate parameters $\mu_i^{(t)}$, for $i = 1, 2$, with the joint evolution of $(\mu_1^{(t)}, \mu_2^{(t)})$ described through the 2d dynamical system in equation \eqref{eq:mu_t}. 
Instead of analyzing the dynamics of the two-dimensional sequential system in \eqref{eq:mu_t}, we achieve a further simplification by decoupling it into the following one-dimensional systems 
\be
\begin{aligned}
    \mu_1^{(t)} &= G\Big(\lambda_2, n \, G\big(\lambda_1, n \mu_1^{(t-1)} \big)\Big), \quad t\geq 2, \\
    \mu_2^{(t)} &= G\Big(\lambda_1, n \, G\big(\lambda_2, n \mu_2^{(t-1)} \big)\Big), \quad t\geq 1.
\end{aligned}
\label{eq:sys_decouple}
\ee
These equations, when initialized using $\mu_1^{(1)} = G\big(\lambda_1,n \mu_2^{(0)} \big)$, will produce the same behavior as the original system \eqref{eq:mu_t} and are far easier to analyze than the coupled two-dimensional system. 

We now show that the system \eqref{eq:sys_decouple} has a {\em unique globally attracting fixed point} $(\mu_1^\star, \mu_2^\star)$, i.e., the updates $\big(\mu_1^{(t)}, \mu_2^{(t)}\big)$ will always converge to $(\mu_1^\star, \mu_2^\star)$ as $t$ tends to infinity, irrespective of the initialization -- there is no possibility of cycles or divergence. The asymmetric ($\lambda_1 < \lambda_2$) and symmetric ($\lambda_1 = \lambda_2$) cases require separate attention, and are treated in two parts. We comment that the existence of unique fixed points is generally not guaranteed; see, for example, \cite{plummer2020dynamics}.

\begin{lemma}\label{lemma:fp}
Let $0<\lambda_1 < \lambda_2 <\infty$. Fix $n \in \mb N$. 
\begin{enumerate}
    \item The function $x \mapsto G\big(\lambda_1,nG(\lambda_2,nx)\big)$ on $[0, \infty)$ has a unique fixed point which lies in the interval $\Lambda_1 :\,= [0, \lambda_1/(\lambda_1 + 1)]$. 
    \item The function $x \mapsto G\big(\lambda_2,nG(\lambda_1,nx)\big)$ on $[0, \infty)$ has a unique fixed point which lies in the interval $\Lambda_2 :\,= [0, \lambda_2/(\lambda_2 + 1)]$. 
\end{enumerate}
Each of these fixed points are globally attracting. \\[1ex]
Let $\lambda_1 = \lambda_2=\lambda$ with $0<\lambda<\infty$. The function $x \mapsto G\big(\lambda,nG(\lambda,nx)\big)$ on $[0, \infty)$ has a unique globally attracting fixed point which lies in the interval $[0,\sqrt{\lambda\slash n}]$.  
\end{lemma}

Defining
\be\label{eq:CAVI_max}
\rho^\star = \rho_1^\star \otimes \rho_2^\star, \quad \rho_i^\star = f_{k_i, h_i, n \mu_i^\star}, \quad i = 1, 2,
\ee
it follows from the above result in conjunction with Scheffe's theorem that the CAVI updates in the density space \eqref{eq:cavi_t} converge to $\rho^\star$ in the total variation norm irrespective of the initialization, i.e., $\lim_{t \to \infty} \|\rho^{(t)} - \rho^\star\|_{\TV} = 0$. Importantly, we can now argue that $\rho^\star$ is a global maximizer of $\Psi_n$, and hence $\ELBO(\m F_\MF) = \Psi_n(\rho^\star)$. To see why this is true, assume in the contrary that there exists $\bar{\rho} = \bar{\rho}_1 \otimes \bar{\rho}_2$ with $\Psi_n(\bar{\rho}) > \Psi_n(\rho^\star)$. Initialize the CAVI updates \eqref{eq:cavi_t} with $\bar{\rho}_1$ and $\bar{\rho}_2$. Since the updates converge to $\rho^\star$, invoke equation \eqref{eq:cavi_gen} to conclude that $\Psi_n(\rho^\star) \ge \Psi_n(\bar{\rho})$, arriving at a contradiction. 

It is evident that $(\mu_1^\star, \mu_2^\star)$ (and hence $\rho^\star$) depend on $n$, although we have not attempted to characterize their dependence on $n$ as yet. Since we now know the identity of a global maximizer of $\Psi_n$ for any $n \in \mb N$, we proceed to characterize the order of the fixed points $\mu_1^\star$ and $\mu_2^\star$ as a function of $n$ in Lemma \ref{lemma:fp_order} below. 

\begin{lemma}\label{lemma:fp_order} 
Let $0<\lambda_1 \le \lambda_2<\infty$. The fixed points of the system \eqref{eq:cavi_t} satisfy $\mu_1^\star \mu_2^\star \asymp 1/n$. Moreover, 
\begin{enumerate}
    \item if $\lambda_1<\lambda_2$, then $\mu_1^\star=O(1)$ and $\mu_2^\star\asymp 1\slash n$. 
    \item if $\lambda_1=\lambda_2$, then $\mu_1^\star=\mu_2^\star$ and both fixed points are of order $1/\sqrt{n}$. 
\end{enumerate}
\end{lemma}

Lemma \ref{lemma:fp_order} shows that the product of $\mu_1^\star$ and $\mu_2^\star$ always decays in the order of $1/n$. In the symmetric case when $\lambda_1 = \lambda_2$, this is achieved by equally splitting the order between the two, while in the asymmetric case, the other extreme is observed. In the asymmetric case, $n \mu_1^\star \asymp n$ and $n \mu_2^\star \asymp 1$, and empirically observing this for various choices of $\lambda_1$ and $\lambda_2$ originally motivated our choice of $\wt{\rho}$ in the proof of Theorem \ref{thm:VB_main}. Moreover, as we shall see below, even though the choice of $\wt{\rho}$ was not optimal in the symmetric case, it did provide the best possible order of the ELBO. 

\begin{figure}[htbp!]
    \centering
    \includegraphics[scale=0.3]{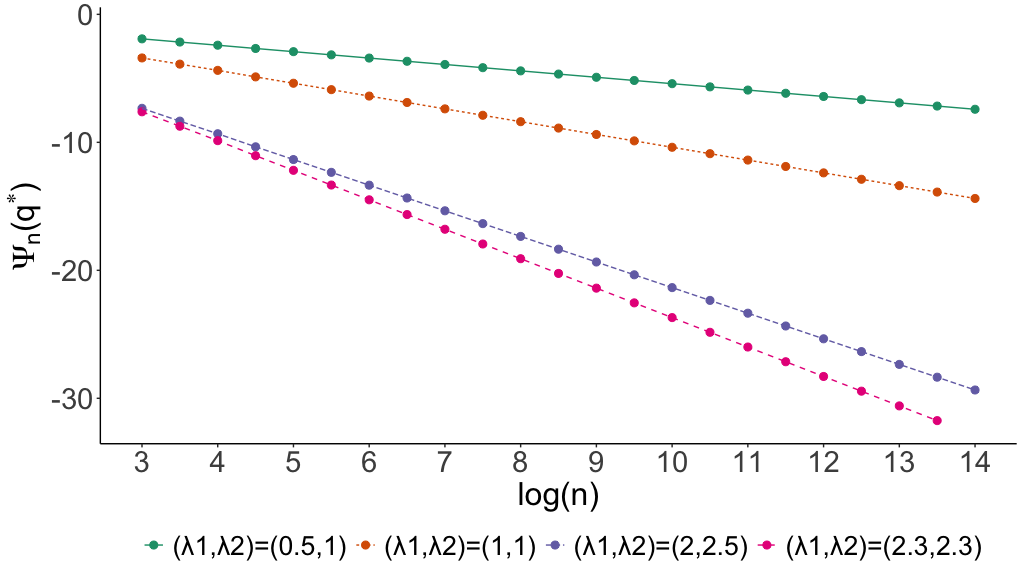}
    \caption{Optimized value of the ELBO for the example in Eq.\,\eqref{eq:special_2d} as a function of $\log n$ for different combinations of $(\lambda_1, \lambda_2)$. }
    \label{fig:scavi_plot2}
\end{figure}

With all ingredients in place, we are now in a position to establish sharpness of our lower bound in Theorem \ref{thm:VB_main}. 
\begin{thm}\label{thm:mf_order}
Consider the normal form of a two parameter singular model with $\mr{h}=(h_1,h_2)$ and $\mr{k}=(k_1,k_2)$. The mean field approximation to the evidence lower bound of normal form is
$$\ELBO(\mathcal{F}_{\MF})=  -\lambda \log n + C$$
for some constant $C$, where $\lambda$ is the RLCT.  
\end{thm}
\begin{proof}
Recall that $\ELBO(\m F_\MF) = \Psi_n(\rho^\star) = \lim_{t \to \infty} \Psi_n(\rho_t)$ with the expression for $\Psi_n(\cdot)$ provided in Eq.\,\eqref{eq:elbo_t}. From Lemma \ref{lemma:fp_order}, we know that $\rho^\star$ corresponds to $\mu_1^\star=O(1)$ \& $\mu_2^\star\asymp 1\slash n$ in the asymmetric case, and $\mu_1^\star = \mu_2^\star \asymp 1/\sqrt{n}$ in the symmetric case. Substituting these values separately in the expression for $\Psi_n(\cdot)$ and using the bounds from Lemma \ref{lem:qkhb} delivers the desired result. 
\end{proof}

Figure \ref{fig:scavi_plot2} empirically confirms the conclusion of Theorem \ref{thm:mf_order} by plotting $\Psi_n(\rho^\star)$ as a function of $\log n$ for different combinations of $(\lambda_1, \lambda_2)$. Specifically, for each of the four pairs of $(\lambda_1, \lambda_2)$, we vary $n$ over a fixed grid of values in the logarithmic scale. For each value of $n$, we run the dynamical system in Eq.\,\eqref{eq:sys_decouple} to convergence, and plot the corresponding converged value of the ELBO from Eq.\,\eqref{eq:elbo_t} against $\log n$. In each case, the points almost exactly line up with a slope of $- \min\{\lambda_1, \lambda_2\}$. 

In Appendix \ref{app:ex27}, we revisit the one-layered neural network example and obtain the coordinate ascent updates both in the original $(\theta_1, \theta_2)$ coordinates and in the transformed $(\xi_1, \xi_2)$ coordinates that render the likelihood into a normal form.  Refer to Eq.\,\eqref{eq:cavi_original} and Eq.\,\eqref{eq:cavi_transformed} and their derivations in Appendix \ref{app:ex27}. The plot of the evidence lower bound in the transformed coordinates (Eq.\,\ref{eq:elbo_transformed}) as a function of sample size $n$ is shown in Figure \ref{fig:scavi_nn}. Recall that $\lambda_1 = \lambda_2 = 1/2$ in this example. The slope of the fitted line is $-1/2$, thus delivering the correct order of the evidence lower bound and once again confirming the conclusion of Theorem \ref{thm:mf_order}. We note here that the CAVI updates in the original untransformed parameterization did not produce the correct order of the ELBO in this example. 

In this example, and in general, the transformation to normal form requires explicit knowledge of the RLCT, which is only known in a handful of settings as indicated earlier. As many singular models have unknown RLCT, it would be worthwhile to develop a general class of transformation-based variational families where we can learn the necessary transformation of the parameters to the normal form before assuming a mean-field structure in the transformed parameterization.


\begin{figure}[htbp!]
    \centering
    \includegraphics[scale=0.5]{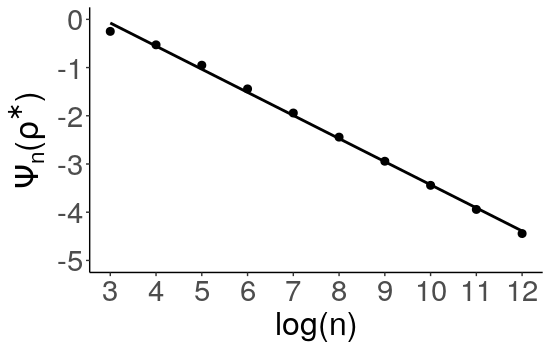}
    \caption{The ELBO as a function of $\log n$ in the one-layered neural network example. }
    \label{fig:scavi_nn}
\end{figure}

\appendix

\section{Remaining proofs from Section 2}\label{app}

\subsection{Proof of Lemma \ref{lem:stoc_order}}
For $G \sim \mbox{Gamma}(\alpha, \lambda)$ and any $t > 0$, 
\be 
P(G \le t) = \frac{\lambda^\alpha}{\Gamma(\alpha)} \int_0^t e^{-\lambda x} x^{\alpha - 1} dx = \frac{1}{\Gamma(\alpha)} \int_0^{\lambda t} e^{-x} x^{\alpha - 1} dx. 
\ee
is an increasing function of $\lambda$ (for fixed $\alpha$ and $t$). Thus, if $Z_i \sim \mbox{Gamma}(\alpha, \lambda_i)$ for $i = 1, 2$ with $\lambda_1 > \lambda_2$, then $Z_1 <_{\st} Z_2$. 

The proof then follows from the fact that if $Z_1, Z_2, Z_3$ are non-negative random variables with $Z_1 <_{\st} Z_2$, then $Z_1 + Z_3 <_{\st} Z_2 + Z_3$.

\subsection{Proof of Corollary 2.1}
Since $b(\cdot)$ is analytic on $U$ containing $\Omega$, we have, for any $\xi \in \Omega$ that 
\be 
b(\xi) = b(\mathbf{0}) + \sum_{|\alpha| \ge 1} \frac{\partial^\alpha b(\mathbf{0})}{\alpha !} \, \xi^{\alpha},
\ee
where $\alpha = (\alpha_1, \ldots, \alpha_d)$ is a multi-index with $|\alpha| = \sum_{j-1}^d \alpha_j$, $\partial^\alpha b = \partial^{\alpha_1} \ldots \partial^{\alpha_d} b$, and $\alpha ! = \alpha_1 ! \ldots \alpha_d !$. Now use the dominated convergence theorem to interchange the integral and sum, and observe that the constant term provides the dominating order. 

\subsection{Proof of Proposition \ref{prop:weak}}
We show that $W_n(\xi)$ converges weakly to the Gaussian process $W^*$. By Theorem 1.5.7 in \cite{van1996weak} it suffices to show the marginal weak convergence and asymptotic tightness of $W_n(\xi)$.  We begin with the convergence of the marginals. For $\xi_1,\ldots,\xi_L\in [0,1]^d$ with integer $L>0$. Applying the multivariate central limit theorem, as $n\to \infty$,
\begin{align*}
    (W_n(\xi_1),\ldots, W_n(\xi_L))\to N(0,C)
\end{align*}
where $C=(c_w(\xi_i,\xi_j))_{1\leq i,j\leq L}$.  Next we show the asymptotic tightness of $W_n(\xi)$ by proving the following three sufficient conditions.  First $[0,1]^d$ is totally bounded. The second condition is the tightness of $W_n(\xi_0)$ for a fixed $\xi_0$. Fix $\xi_0\in [0,1]^d$, for $\epsilon >0$. We need to show that there exists a compact set $K$, such that $\mathbb{P}\{W_n(\xi_0)\in K\}>1-\epsilon$. We construct $K=\{\vert W_n(\xi_0)\vert \leq t\}$ with $t$ chosen as follows. By {\bf Assumption A1}, $\tilde{Z}_i(\xi_0)$ are independent centered sub-Gaussian and by Theorem 2.6.2  of \cite{vershyninhigh}
\begin{align*}
    \mathbb{P}(\vert W_n(\xi_0) \vert \geq t) \leq \mathbb{P}\left(\Big\vert \sum_{i=1}^n\tilde{Z}_i(\xi_0)\Big\vert \geq \sqrt{n}t\right)\leq 2\exp(-ct^2)
\end{align*}
for some constant $c>0$. Choosing $t=\sqrt{2\log(1\slash \epsilon)}$ completes the proof of tightness. 

The third condition is that $W_n(\xi)$ is asymptotically uniformly $d$-equicontinuous, where $d(\xi,\zeta)=\Vert \xi-\zeta \Vert$ is the metric generated by the norm in {\bf Assumption A1}. $W_n(\xi)$ is said to be asymptotically uniformly $d$-equicontinuous if for any $\eta,\epsilon>0$, there exists a $\delta>0$ such that 
\begin{align*}
    \mathbb{P}\left\{ \sup_{d(\xi,\zeta)<\delta} \vert W_n(\xi) -W_n(\zeta)\vert >\epsilon \right\} <\eta. 
\end{align*}
To that end, 
\begin{align*}
    \sup_{d(\xi,\zeta)<\delta}\vert W_n(\xi) -W_n(\zeta)\vert \leq \frac{1}{\sqrt{n}}\sum_{i=1}^n\vert\tilde{Z}_i(\xi)-\tilde{Z}_i(\zeta)\vert \leq   \frac{\delta}{\sqrt{n}}\sum_{i=1}^nL(X_i)
\end{align*}
and 

\begin{align*}
    \mathbb{P}\Big\{ \sup_{d(\xi,\zeta)<\delta} \vert W_n(\xi) -W_n(\zeta)\vert > \epsilon\Big\} & \leq \mathbb{P}\Big\{ \sum_{i=1}^nL(X_i) > \frac{\sqrt{n}\epsilon}{\delta}\Big \} \\ &= 
  \mathbb{P}\Big[ \exp\Big\{t\sum_{i=1}^nL(X_i)\Big\} > \exp \Big(t \frac{\sqrt{n}\epsilon}{\delta}\Big)\Big] \\
    & \leq \exp\big\{-t \sqrt{n}\epsilon/\delta + n t c_L^2/2\big\}
\end{align*}
for any $t > 0$, where the final inequality follows from Markov's and {\bf Assumption A1}.   Setting $t = \epsilon /(\delta \sqrt{n} c_L^2)$, we obtain 
\begin{align*}
    \mathbb{P}\Big\{ \sup_{d(\xi,\zeta)<\delta} \vert W_n(\xi) -W_n(\zeta)\vert > \epsilon\Big\} \leq  e^{- \epsilon^2/(2 c_L^2 \delta^2)}. 
\end{align*}
Choosing $\delta= \epsilon/(c_L\sqrt{2\log (1/\eta)})$ completes the proof of asymptotically uniformly $d$-equicontinuous. 
Therefore the conditions of Theorem 1.5.7 in \cite{van1996weak}  are met and $W_n(\xi)$ converges weakly to a Gaussian process. 

\subsection{Proof of Proposition \ref{lem:Dnt_conv}}
We shall prove only the second part; the proof of the first part is very similar and is omitted. 
We use the notation
\be
G(t, W) := t^{\lambda -1} e^{-t - \sqrt{t}W}. 
\ee
Observe that $\int_0^n |D_n(t)  - D(t) |dt $ is bounded above by 
\be
 &\int_{0}^n \int |G(t, W^*({\bf 0}, \xi_J))  -  G(t, W^*(\xi_I, \xi_J)) |  \varphi_{\xi \mid Z= -\log (t/n)} (\xi_I)  \varphi(\xi_J) d\xi dt+ \nonumber \\  
& \int_{0}^n \int |G(t, W_n(\xi_I, \xi_J))  -  G(t, W^*(\xi_I, \xi_J)) |  \varphi_{\xi \mid Z= -\log (t/n)} (\xi_I)  \varphi(\xi_J) d\xi dt\label{eq:diff}
\ee
To control the first term, for given any $\epsilon> 0$, there exists $\delta > 0$, such that  $\sup_{\{\|\xi_I\| < \delta \}} |e^{-\sqrt{t} W^*({\bf 0}, \xi_J)}  -  e^{-\sqrt{t} W^*(\xi_I, \xi_J)} | < \epsilon$.  Then the first term is less than 
\be
\epsilon + 2 \int_{0}^n t^{\lambda -1}e^{-t + \sqrt{t} \|W^* \|_\infty} \Big[ \int_{\{\|\xi_I\| \geq \delta \}}  \varphi_{\xi \mid Z= -\log (t/n)} (\xi_I)  d\xi_I  \Big] dt \label{eq:diff}
\ee
Observe that the second term in the r.h.s of \eqref{eq:diff}  is bounded above by 
\be
(m-1) \bbP \{\xi_1 > \delta /\sqrt{m-1} \mid Z= -\log (t/n)\}.
\ee
The one dimensional marginal $\xi_1 \mid Z$  of the conditional density \eqref{eq:cond} is given by 
\be
\varphi_{\xi_1 \mid Z}(\xi_1) = \frac{2 k_1}{\xi_1 Z} , \quad e^{-Z}\leq \xi_1^{2k_1} \leq 1.  
\ee
Note that the sequence of random variables 
$f_n(\xi_1) = \ind (\xi_1 > \delta /\sqrt{m-1})  \varphi_{\xi_1 \mid Z= -\log (t/n)} (\xi_1)$ converges to zero and is bounded above by the integrable function $\varphi_{\xi_1 \mid Z= -\log (t/n)} (\xi_1)$, hence an application of the dominated convergence theorem shows $\int f_n(\xi_1) d\xi_1$ converges to 0.  Another application of DCT shows that the second term in \eqref{eq:diff} converges to 0. 

The second term of \eqref{eq:diff} can be bounded above by 
\be
\int_{0}^\infty t^{\lambda -1} e^{-t} \int |e^{\sqrt{t}W_n(\xi_I, \xi_J)}  - e^{\sqrt{t}W^*(\xi_I, \xi_J)} |  \varphi_{\xi \mid Z= -\log (t/n)} (\xi_I) \varphi(\xi_J)  d\xi
\ee
Since $W_n \stackrel{w}{\rightarrow} W^*$, by continuous mapping,  the above converges weakly to $0$. 
Finally, 
\be
\int_n^\infty D(t) dt &= \int_{t=n}^\infty t^{\lambda -1} e^{-t - \sqrt{t}W^*({\bf 0}, \xi_J)}  \varphi_{\xi \mid Z= -\log (t/n)} (\xi_I) \varphi(\xi_J)  d\xi_I dt \\
&\leq \int_{t=n}^\infty t^{\lambda -1} e^{-t + \sqrt{t} \|W^*\|_\infty} dt 
\ee
which converges to $0$, concluding the proof. 

\section{Remaining proofs from Section 3}\label{app2}

\subsection{Proof of Lemma \ref{lem:qkhb}}
Part (i) follows from a change of variable $v = \beta u^{2k}$. For part (ii), we have, using the definition of $B(k,h, \beta)$,
\be 
G(\lambda,\beta) = \frac{B(k,2k+h,\beta)}{B(k,h,\beta)}  = \frac{ (2k)^{-1} \beta^{-(\lambda+1)} \Gamma(\lambda+1) \gamma(\lambda+1,\beta) }{ (2k)^{-1} \beta^{-(\lambda)} \Gamma(\lambda) \gamma(\lambda,\beta) } = \frac{\lambda}{\beta} \, \frac{\gamma(\lambda+1,\beta)}{\gamma(\lambda, \beta)}.
\ee
For the first part of part (iii), we have $\log B(k, h, \beta) = - \lambda \log \beta + \log \gamma(\lambda, \beta) + \text{terms free of } \beta$. The conclusion follows since $\lim_{\beta \to \infty} \gamma(\lambda, \beta) = 1$. 

For the second part of part (iii), use Remark \ref{rem:incomp_gamma} to write 
\be 
G(\lambda, \beta) =\frac{\lambda}{\beta}\left(1-\frac{\beta^\lambda e^{-\beta}}{\Gamma(\lambda+1)\gamma(\lambda,\beta)}\right). 
\ee
From the above, the conclusion follows since $\lim_{\beta \to \infty} \beta^\lambda e^{-\beta} = 0$ and $\lim_{\beta \to \infty} \gamma(\lambda, \beta) = 1$. 

\subsection{Proof of Lemma \ref{lemma:fp}}

\noindent {\bf Asymmetric case ($\lambda_1 < \lambda_2$). \,} We prove the first statement, the proof of the second is similar. Consider the function $F(x)=G(\lambda_1,nG(\lambda_2,nx))-x$. We have $F(0)=  G(\lambda_1,nG(\lambda_2,0))> 0$ and $F(\lambda_1\slash(\lambda_1+1))<0$ since $G(\lambda_1,nG(\lambda_2,nx))<\lambda_1\slash(\lambda_1+1)$.  
Applying the intermediate value theorem to the function $F$ on the interval $\Lambda_1$ gives the result. 
Notice for $x > \lambda_1\slash(\lambda_1+1)$ the function $G(\lambda_1,nG(\lambda_2,nx))-x<0$ and cannot have a root in $(0,\infty)\setminus \Lambda_1$. 

\noindent {\bf Symmetric case ($\lambda_1 = \lambda_2$). \,} We first build state some preparatory results. 

\begin{lemma}
Suppose $f:(0,\infty) \to (0,\infty)$ is strictly monotone on $(0,\infty)$, then $f$ has no periodic points of period $p>1$.
\end{lemma}

\begin{lemma}\label{lemma:dynamic.constraint}
The discrete dynamical system defined by 
\begin{align}
\begin{split}
    \mu_1(k)&=f^{2k-1}(\mu_0)\\
    \mu_2(k)&=f^{2k}(\mu_0)
\end{split}
\end{align}
can only converge to a fixed point of $x_0$ of $f$ or a two cycle $C=\{x_1,x_2\}$ of $f$. Here $f^k(x)$ denotes the $k$-fold composition of $f$ evaluated at $x$. 
\end{lemma}

\begin{lemma}\label{lemma:dynamics.G}
For $\lambda>0$, function $G(\lambda, nx)$ has a single fixed point $x_*$ in the interval $[0, \sqrt{\lambda\slash n}]$ and no other fixed points in $[0,\infty)$. 
\end{lemma}
\begin{proof}
Suppose $\lambda>0$. The $G(\lambda,nx)$ is monotone decreasing on $[0,\infty)$. Its maximum is $\lambda\slash (\lambda+1) = G(\lambda, 0)>0$. There must be a root of $G(\lambda, nx)-x$ on the interval $[0,\lambda\slash (\lambda+1)]$ by intermediate value theorem. Similarly $G(\lambda, nx)-x<0$ on the interval $(\lambda\slash(\lambda+1), \infty)$ so it has no roots on this interval. 

We now provide a more detailed bound on the fixed point. It follows from Lemma \ref{lem:qkhb} that $G(\lambda,nx)-x < \lambda\slash (nx) - x$. Letting $x_*$ denote the fixed point of $G(\lambda,nx)$ in $[0,\lambda\slash(\lambda+1)]$ we have 
\begin{align}
0=G(\lambda,nx_*)-x_* < \lambda\slash (nx_*) - x_*
\end{align}
from which it follows that $(x_*)^2 < \lambda\slash n$ so $x_*< \vert \sqrt{\lambda\slash n}\vert$.  Therefore $0<x_*<\sqrt{\lambda\slash n}$.
\end{proof}

\

We now complete the proof of the symmetric case. To simplify notation let $\lambda=\lambda_1=\lambda_2$ and $F(x)=G(\lambda,nx)$. Let $f^k(x)$ denote the $k$-fold composition of $f$. The system in the symmetric case is 
\begin{align}\label{sys:symmetric}
    \begin{split}
        \mu_2(0)&=\mu_0\\
        \mu_1(k)&= F^{2k-1}(\mu_0) \\
        \mu_2(k)&= F^{2k}(\mu_0)
    \end{split}
\end{align}

The structure of the symmetric system heavily limits the possible convergence behavior of the system. Lemma \ref{lemma:dynamic.constraint} shows that this type of system can only converge to a fixed point of $F$ or to 2-cycle of $F$. Furthermore, lemma \ref{lemma:dynamics.G} shows that $F(x)=G(\lambda,nx)$ only has a single fixed point $x_*$ which asymptotically approaches $\sqrt{\lambda\slash n}$ and no periodic points of order $p>1$. This completes the proof. 

\subsection{Proof of Lemma \ref{lemma:fp_order}}
The proof for the symmetric case has already been established in the previous Lemma. We therefore focus on the asymmetric case. We begin with recording a useful result. We will need to know if there exist solutions to certain equations which arise when studying the order of the fixed points for the system \eqref{eq:sys_decouple}. More specifically these equations are $zG(\alpha, z)=\beta$ and $1= xG(\alpha, \beta x)$. They are equivalent under the change of variables $\beta x=z$, The existence of solutions to these equations is as follows. 
 \begin{lemma} \label{lemma:constraint.eq.1}
 Suppose $\alpha,\beta\in (0,\infty)$. Recall $G$ defined as in Lemma \ref{lem:qkhb}.
 \begin{enumerate}
     \item For $\alpha>\beta$, $zG(\alpha, z)=\beta$ has a solution in $(0,\infty)$.
     \item For $\alpha\leq \beta$, $zG(\alpha, z)=\beta$ has no solution in $(0,\infty)$.
 \end{enumerate} 
 \end{lemma}
\begin{proof}
 1. If $\alpha>\beta$, apply intermediate value theorem to $zG(\alpha, z)-\beta$ which converges to  $\alpha-\beta>0 $ as $z\to \infty$ and starts at $-\beta$ for $z=0$. So the equation has a root. \\
 2. If $\alpha\leq \beta$, apply intermediate value theorem to $zG(\alpha, z)-\beta$ which converges to $\alpha-\beta \leq 0$ as $z\to \infty$ and starts at $-\beta$ for $z=0$. So the equation has no root in $(0,\infty)$.
\end{proof}

Back to the main proof, write $\mu_1^\star=c_1\slash f(n)$ and $\mu_2^\star=c_2\slash f(n)$, where $c_1,c_2$ are constants that do not depend on $n$ and $f:\mathbb{N}\to (0,\infty)$. We will analyze the behavior of the order of the fixed point $\mu_1^\star$. The analysis of $\mu_2^\star$ is similar. The fixed point equation becomes
\begin{align}
\mu_1^\star&= G(\lambda_2, nG(\lambda_1, n\mu_1^\star)) \nonumber\\
\frac{c_1}{f(n)} &= G\left(\lambda_2, nG\left(\lambda_1, c_1 \frac{n}{f(n)}\right)\right) 
\end{align}
\textbf{Case 1} Assume that $\lim_{n\to \infty}f(n)\slash n =\infty$. Then $\lim_{n\to \infty}n\slash f(n)=0$ and for sufficiently large $n$ we have 
\begin{align}
\frac{c_1}{f(n)} &= G\left(\lambda_2, nG\left(\lambda_1, c_1 \frac{n}{f(n)}\right)\right) \nonumber \\
&\approx  G\left(\lambda_2, n\frac{\lambda_1}{\lambda_1+1} \right) \approx \frac{\lambda_2(\lambda_1+1)}{\lambda_1 n}
\end{align}
This implies $f(n)=O(n)$, a contradiction of our assumptions. \\
\\
\textbf{Case 2} Without loss of generality we can assume that $\lim_{n\to \infty}f(n)\slash n =1$. Then $\lim_{n\to \infty}n\slash f(n)=1$ and for sufficiently large $n$ we have 
\begin{align}
\frac{c_1}{f(n)} &= G\left(\lambda_2, nG\left(\lambda_1, c_1 \frac{n}{f(n)}\right)\right) \nonumber \\
&\approx G(\lambda_2, nG(\lambda_1, c_1)) \approx \frac{\lambda_2}{ nG(\lambda_1, c_1)}
\end{align}
The constant $c_1$ is a solution to the additional constraining equation
\begin{align}
    zG(\lambda_1, z)=\lambda_2 \label{eqn:c1.constraint}
\end{align}
Applying an analogous argument to $\mu_2^\star$, we get that $c_2$ must solve the constraining equation
\begin{align}
    zG(\lambda_2, z)=\lambda_1 \label{eqn:c2.constraint}
\end{align}
 By Lemma \ref{lemma:constraint.eq.1} there is a solution to the equation (\ref{eqn:c2.constraint}) for $c_2$, but still no solution to the equation (\ref{eqn:c1.constraint}) for $c_1$. 
 
\textbf{Case 3} Assume that $\lim_{n\to \infty}f(n)\slash n =0$ and $\lim_{n\to\infty} f(n)=\infty$. Then $\lim_{n\to \infty}n\slash f(n)=\infty$ and for sufficiently large $n$ we have 
\begin{align}
\frac{c_1}{f(n)} &= G\left(\lambda_2, nG\left(\lambda_1, c_1 \frac{n}{f(n)}\right)\right) \approx G\left(\lambda_2, \lambda_1 f(n)\right) \approx \frac{\lambda_2}{\lambda_1}\frac{c_1}{f(n)}
\end{align}
Under the assumption $\lambda_1<\lambda_2$ we have reached a contradiction. 

\textbf{Case 4} Assume that $\lim_{n\to \infty}f(n)\slash n =0$ and $f(n)$ is bounded. This implies that $\mu_1^\star\approx 1\slash c_1$, where $c_1$ is a constant independent of $n$, 
\begin{align} 
\frac{1}{c_1} &= G\left(\lambda_2, nG\left(\lambda_1, \frac{n}{c_1}\right)\right) \approx G(\lambda_2, \lambda_1 c_1)
\end{align}
By Lemma \ref{lemma:constraint.eq.1}, the equation for $1= zG(\lambda_1, \lambda_2z)$ which determines $c_1$ has a solution but  the equation $1= zG(\lambda_1, \lambda_2z)$ which determines $c_2$ has no solution.

These four cases exhaust all the possible behaviors of a generic $f:\mathbb{N}\to (0,\infty)$. Therefore we have exhausted all possible solutions bounded fixed points of the functions $G(\lambda_2,nG(\lambda_1,nx))$ and $G(\lambda_1,nG(\lambda_2,nx))$. For $\lambda_1<\lambda_2$ the only possible solutions to the additional constraining equations yeild $\mu_1^\star=c_1$ where $c_1$ is the solution to $1=zG(\lambda_2,\lambda_1z)$ and $\mu_2^\star=c_2\slash n$ where $c_2$ is the solution to $\lambda_1=zG(\lambda_2,z)$. This completes the proof of order in the asymmetric case.

\section{Derivation of CAVI and ELBO for Example 27 in Watanabe's book} \label{app:ex27}
%
The CAVI updates in the original $(\theta_1, \theta_2)$ coordinate is given by 
\begin{equation}\label{eq:cavi_original}
\begin{aligned}
    \log q_1(\theta_1) &=  -\frac{n K_2}{2}\left(\theta_1-\frac{K_1}{K_2}\right)^2 -\frac{n}{2}S_{yy}+\frac{nK_1^2}{2K_2} + \mbox{Const.}\\
        \log q_2(\theta_2) &= -\frac{n m_2}{2}T_2(\theta_2, X^n) + nm_1T_1(\theta_2,X^n,Y^n)- \frac{n}{2}S_{yy}
    \end{aligned}
    \end{equation}
where $m_1=\int_0^1\theta_1q_1(\theta_1)d\theta_1$, $m_2 =\int_0^1\theta_1^2q_1(\theta_1)d\theta_1$, $S_{yy}= \sum_{i=1}^n y_i^2$ and
    \begin{align*}
    \begin{aligned}
               T_1(X^n,Y^n,\theta_2)&=\frac{1}{n}\sum_{i=1}^n y_i\tanh(\theta_2x_i) \\
                  T_2(X^n,\theta_2) &= \frac{1}{n}\sum_{i=1}^n \tanh^2(\theta_2x_i)    
           \end{aligned}
     \quad
     \begin{aligned}
         K_1(X^n,Y^n)&=\int_{0}^1T_1(X^n,Y^n,\theta_2)q_2(\theta_2)d\theta_2 \\
        K_2(X^n)&=\int_{0}^1T_2(X^n,\theta_2)q_2(\theta_2)d\theta_2. 
        \end{aligned}
            \end{align*}
It follows that the ELBO is given by   $I_p + I_1 + I_2$, where 
\begin{align*}
I_p &=  \frac{n}{2} \log(2\pi)-\frac{n}{2}\left(S_{yy} -2m_1K_1(X^n,Y^n)+m_2K_2(X^n)\right), \\
I_1 &=  - \log\mathcal{Z}_1+\frac{nK_2}{2}\left(m_2-2\frac{K_1}{K_2}m_1+\frac{K_1^2}{K_2^2}\right),\\
I_2 &=  \log\mathcal{Z}_2 +\frac{nm_2}{2}K_2(X^n)  - nm_1K_1(X^n,Y^n). 
\end{align*}
In the following, we shall transform coordinates to reduce the likelihood to a normal form.  To aid this, we 
rewrite the likelihood function as
\begin{align*}
    p(Y^n\mid X^n, \theta_1,\theta_2)&=\frac{1}{\log(2\pi)^{n\slash 2}}\exp\left\{ -\frac{n}{2}S_{yy}-\frac{n}{2}\theta_1^2\theta_2^2K_{n,0}(\theta_2,X^n)-n\theta_1\theta_2K_{n,1}(\theta_2,X^n,Y^n)\right\}
\end{align*}
where 
\begin{align*}
K_{n,0}(\theta_2, X^n)=\frac{1}{n}\sum_{i=1}^n\left\{\frac{\tanh(\theta_2x_i)}{\theta_2}\right\}^2,  \quad  K_{n,1}(\theta_2,X^n,Y^n)=\frac{1}{n}\sum_{i=1}^n\left\{\frac{y_i\tanh(\theta_2x_i)}{\theta_2}\right\}. 
\end{align*}
We use the following asymptotic limits of  $K_{n,0}$ and $K_{n,1}$ to approximate and thus simplify the likelihood further. 
\begin{align*}
        K_{n,1} \to \frac{1}{\theta_2}\mathbb{E}\left(\mathbb{E}[Y\tanh(\theta_2X)\mid X]\right)=0, \quad K_{n,0} \to K_0(\theta)=\int_0^1\frac{\tanh(\theta_2x)}{\theta_2}dx. 
\end{align*}
Therefore the posterior $p(\theta_1,\theta_2\mid X^n,Y^n)$ concentrates around 
\begin{align*}
         \gamma_K^{(n)}(\theta_1,\theta_2) \,\propto \, \exp\left\{-nK(\theta_1,\theta_2)\right\}\phi(\theta_1,\theta_2),  \quad K(\theta_1,\theta_2)= \frac{1}{2}\theta_1^2\theta_2^2K_{0}(\theta_2). 
\
\end{align*}
 Conversion to normal form is done by the change of variables 
\begin{equation*}
    \xi_1=\theta_1,\quad \xi_2=\theta_2(K_{0}(\theta_2)\slash2)^{1\slash 2}:=g(\theta_2).
\end{equation*}
The inverse transform is $\theta_1=\xi_1$ and $\theta_2=g^{-1}(\xi_2)$ and its Jacobian is 
\begin{align*}
    J(\xi_2)=\left\vert\frac{1}{g'(g^{-1}(\xi_2))}\right\vert. 
\end{align*}
Under the $\xi$-coordinate system we have 
\begin{align*}
   K(\theta_1,\theta_2)&= K(\xi_1,g^{-1}(\xi_2))=\theta_1^2\theta_2^2\frac{K_{0}(\theta_2, X^n)}{2}=\xi_1^2\xi_2^2. 
\end{align*}
The concentrated normal form of the posterior is 
\begin{align*}
    \gamma_K^{(n)}(\xi_1,\xi_2)  = \exp(-n\xi_1^2\xi_2^2 + \log J(\xi_2) -\log C_L),
\end{align*}
where $C_L$ is the normalizing constant. Then 
\begin{align}\label{eq:cavi_transformed}
    q_1(\xi_1) =  \exp\left\{  -nF_2\xi_1^2 -\log C_1 \right\}, \quad  q_2(\xi_2) = \exp\left\{-nF_1\xi_2^2 + \log J(\xi_2) -\log C_2\right\}, 
    \end{align}
where 
\begin{align*}
        F_1=\int_{0}^{1}\xi_1^2q_1(\xi_1)d\xi_1, \quad   F_2=\int_{g(0)}^{g(1)}\xi_2^2q_2(\xi_2)d\xi_2. 
\end{align*}
The negative KL divergence between the variational distribution $q(\xi)=q_1(\xi_1)\otimes q_2(\xi_2)$ and the normal form of the posterior $\gamma_K^{(n)}(\xi)$ is 
\begin{align*}
     -D(q(\xi) \mid\mid\gamma^{(n)}_K(\xi)) = I_1 + I_2 + I_\gamma,
\end{align*}
where 
\begin{align*}
    I_\gamma 
    = -nF_1F_2 - \log(C_L),  \quad   I_1 =  nF_1F_2  + \log(C_1), \quad 
    I_2 = nF_1F_2  +   \log(C_2). 
\end{align*}
Hence the  optimal evidence lower bound is 
\begin{align}\label{eq:elbo_transformed}
     \Psi_n(q^*) &=  -nF_1^*F_2^* 
     +\log (C_1^*) + \log(C_2^*). 
\end{align}

\section{Verifying Assumption A1 for Example 27 in Watanabe's book}\label{sec:A1verify}
We bound the difference $\vert \tilde{Z}_i(\xi)-\tilde{Z}_i(\xi')\vert$ by repeated application of the triangle inequality, 
\be
\vert \tilde{Z}_i(\xi)-\tilde{Z}_i(\xi')\vert  & \leq A + B + C, 
\ee
where 
\be
A &= \frac{1}{2}\vert \xi_1g^{-1}(\xi_2)F^2(X_i; g^{-1}(\xi_2)) - \xi_1'g^{-1}(\xi_2')F^2(X_i; g^{-1}(\xi_2')) \vert,\\
B &= \vert Y_iF(X_i; g^{-1}(\xi_2)) - Y_iF(X_i; g^{-1}(\xi_2'))\vert, \\
C &= \frac{1}{2} \vert \xi_1g^{-1}(\xi_2)K_0(g^{-1}(\xi_2))- \xi_1'g^{-1}(\xi_2')K_0(g^{-1}(\xi_2'))\vert. 
\ee
Denoting $\|x \|_1 = \sum_{j=1}^d |x_j|$ for $x \in \mathbb{R}^d$, using Lemmata \ref{lemma:F.Lip} and \ref{lemma:K0.Lip} yields
\be 
A&\leq \frac{1}{2}(L_{\infty}^{1,2}(X_i) + C_{g^{-1}}L_{\infty}^{0,2}(X_i) +  C_{g^{-1}}\Vert g^{-1}\Vert_\infty L_2(X_i))\Vert \xi-\xi' \Vert_1 :=L_a(X_i)\Vert \xi-\xi' \Vert_1\\ 
B &\leq  C_{g^{-1}}L_1(X_i)\vert Y_i \vert\Vert \xi-\xi'\Vert_1:=L_b(X_i,Y_i)\Vert \xi-\xi' \Vert_1 \\
C&\leq \frac{1}{2}(\Vert g^{-1}\Vert_\infty \Vert K_0 \Vert_\infty + C_{g^{-1}}\Vert K_0 \Vert_\infty + C_{K_0}C_{g^{-1}} \Vert g^{-1}\Vert_\infty)\Vert \xi-\xi' \Vert_1:=L_c(X_i)\Vert \xi-\xi' \Vert_1. 
\ee
Combining the bounds we have 
 $\vert \tilde{Z}_i(\xi)- \tilde{Z}_i(\xi')\vert \leq  L(X_i) \|\xi - \xi'\|_1$ with with $L(X_i) = L_a(X_i)+L_b(Y_i,X_i)+L_c(X_i)$. 
It remains to show that $L(X_1)$ satisfy the tail bound $\mb Ee^{t L(X_1)} \leq e^{t^2/(2c_L)}$. We will show that $L_a(X_1)$, $L_b(Y_1,X_1)$, $L_c(X_1)$ are all sub-Gaussian, which implies the above bound on $L(X_1)$ holds. $L_c(X_1)$ is a constant random variable so it is sub-Gaussian. $L_b(Y_1,X_1)$ is the product of a bounded random variable $L_1(X)$ with a sub-Gaussian random variable $\vert Y_1\vert$. $L_a(X_1)$ is the sum of bounded random variables and hence sub-Gaussian.

\begin{lemma}\label{lemma:F.Lip}
Suppose $X\sim U[0,1]$. Let $F(X;\omega):=\tanh(\omega X)\slash \omega,\, x,\omega \in [0,1]$, then for $s=1,2$ there exist bounded random variables $L_{\infty}^{0,s}(X)$, $L_{\infty}^{1,s}(X)$, $L_s(X)$ such that
\be 
\sup_\omega\{ F^s(X;\omega)\} &\leq L_{\infty}^{0,s}(X)\\
\sup_\omega\{ \omega F^s(X;\omega)\} &\leq L_{\infty}^{1,s}(X)\\
\vert F^s(X;\omega)-F^s(X;\omega')\vert &\leq L_s(X) \vert \omega-\omega' \vert
\ee
where $L_{\infty}^{k,s}(X)=L_s(X)=\vert X \vert$ for $k=0,1$ and $s=1,2$
\end{lemma}
\begin{proof}
For $s=1$, we have $F^s(X;\omega)= \tanh(\omega X)\slash \omega$ and $\omega F^s(X;\omega)= \tanh(\omega X)$. Both of these functions are globally $\vert\cdot\vert$-Lipschitz functions in the $\omega$-variable since they both have bounded, continuous derivatives. 
\be 
\frac{d}{d\omega}F(X;\omega)&= \frac{d}{d\omega} \left[\frac{\tanh(\omega X)}{\omega}\right]= \frac{X\omega\text{sech}^2(X\omega)-\tanh(X\omega)}{\omega^2}. 
\ee 
It follows 
\be 
\sup_{\omega\in[0,1]}\vert F(X;\omega)\vert &\leq \vert X \vert:= L_{\infty}^{0,1}(X), \\
\sup_{\omega\in[0,1]}\vert \omega F(X;\omega)\vert &=\vert \tanh(X)\vert\leq \vert X\vert := L_{\infty}^{1,1}(X), \\
 \sup_{\omega\in[0,1]}\left \vert \frac{d}{d\omega} \left[\frac{\tanh(\omega X)}{\omega}\right] \right \vert &=  \sup_{\omega\in[0,1]} \left \vert X\omega\left( \frac{\text{sech}^2(X\omega)-\frac{\tanh(X\omega)}{X\omega}}{\omega^2} \right) \right \vert \leq \vert X\vert := L_1(X). 
\ee 
Similarly for $s=2$, we have $F^s(X;\omega)= \tanh^2(\omega X)\slash \omega^2$ and $\omega F^s(X;\omega)= \tanh^2(\omega X)\slash \omega$. Both of these functions are globally $\vert\cdot\vert$-Lipschitz functions in the $\omega$-variable since they both have bounded, continuous derivatives. 
\be 
\frac{d}{d\omega}F(X;\omega)&= \frac{d}{d\omega} \left[\frac{\tanh^2(\omega X)}{\omega^2}\right]= \frac{\tanh(X\omega)(\omega\text{sech}^2(X\omega)-2\tanh(X\omega)}{\omega^3}. 
\ee 
It follows 
\be 
\sup_{\omega\in[0,1]}\vert F^2(X;\omega)\vert &\leq \vert X\vert:= L_{\infty}^{0,2}(X), \\
\sup_{\omega\in[0,1]}\vert \omega F^2(X;\omega)\vert &\leq \vert X \vert:= L_{\infty}^{1,2}(X), \\
 \sup_{\omega\in[0,1]}\left \vert \frac{d}{d\omega} \left[\frac{\tanh^2(\omega X)}{\omega}\right] \right \vert &=  \sup_{\omega\in[0,1]} \left \vert 2X\omega\left( \frac{\text{sech}^2(X\omega)-\frac{\tanh(X\omega)}{X\omega}}{\omega^3} \right) \right \vert \leq \vert X\vert := L_2(X).
\ee 
\end{proof}
\begin{lemma}\label{lemma:K0.Lip}
The functions
$$K_0(\omega)=\int_0^1 \frac{\tanh^2(\omega x)}{\omega^2}dx,\,\omega\in[0,1]$$ and the inverse of the function $g(\theta)=\theta(K_0(\theta)\slash \theta)^{1\slash 2}$, for $\theta\in[0,1]$, are  globally $\vert \cdot \vert$-Lipschitz with Lipschitz constants $C_{K_0}$ and $C_{g^{-1}}$,
respectively,
\be
\vert K_0(\omega)-K_0(\omega')\vert \leq C_{K_0}\vert \omega -\omega' \vert, \quad \vert g^{-1}(\omega)-g^{-1}(\omega')\vert \leq C_{K_0}\vert \omega -\omega' \vert. 
\ee
\end{lemma}
\begin{proof}
The function $K_0$ has a bounded derivative for $\omega\in [0,1]$. It follows that $K_0$ is globally Lipschitz on $[0,1]$ with Lipschitz constant $C_{K_0}=\sup_{\omega}\vert K_0'(\omega) \vert$. Also, since $g'$ is bounded in $[0, 1]$ and for 
$\omega=g(\theta)$, $(d/d \omega)g^{-1}(\omega) = 1/g'(g^{-1}(\omega))$, 
the function $g^{-1}$ has a bounded derivative in $[0, 1]$. It follows that $g^{-1}$ is globally Lipschitz on $[0,1]$ with Lipschitz constant $C_{g^{-1}}=\sup_{\omega}\vert \frac{d}{d\omega}g^{-1}(\omega) \vert.$
\end{proof}



\bibliographystyle{plainnat}
\bibliography{sing,sts_singl_additional_refs}


%
%
%
\end{document}